\documentclass[a4paper,12pt]{amsart}
\newfont{\cyr}{wncyr10 scaled 1100}

\usepackage[left=2.7cm,right=2.7cm,top=3.5cm,bottom=3cm]{geometry}

\usepackage{amssymb,amsmath,latexsym,braket,amsthm, multirow,xcolor,mathrsfs}
\usepackage{amssymb,amsmath,xcolor,textcomp}
\usepackage{amsthm}
\usepackage{url,bbm,extarrows,setspace}
\usepackage{csquotes}
\usepackage{hyperref}
\usepackage{tikz,subcaption}
\usepackage[enableskew]{youngtab}
\usepackage{ytableau,varwidth}
\usepackage{pst-node}
\usepackage{rotating}
\usepackage{amsfonts,amsmath,amssymb,amsthm,mathtools}
\usepackage{accents,dsfont,enumitem,graphicx}
\usepackage{array,hyperref,mathrsfs,url}
\usepackage{caption,subcaption}

\usepackage{tikz-cd} 
\usetikzlibrary{arrows}

\newtheorem{theorem}[table]{Theorem}

\newtheorem{proposition}{Proposition}[section]

\numberwithin{table}{section}

\newcommand{\Hp}{\mathcal{H}_p}
\newcommand{\Z}{\mathbb{Z}}

\newcommand{\Q}{\mathbb{Q}}
\newcommand{\C}{\mathbb{C}}
\newcommand{\R}{\mathbb{R}}
\newcommand{\PQp}{\mathbb{P}^1(\mathbb{Q}_p)}
\newcommand{\PCp}{\mathbb{P}^1(\mathbb{C}_p)}

\newcommand{\T}{\mathcal{T}}

\DeclareMathOperator{\MS}{MS}
\DeclareMathOperator{\SL}{SL}
\DeclareMathOperator{\ST}{ST}
\DeclareMathOperator{\Res}{Res}

\DeclareMathOperator{\Id}{Id}

\newcommand{\Aui}{A_{Q,1}^{(i)}}
\newcommand{\Aki}{A_{Q,k}^{(i)}}
\newcommand{\Ali}{A_{Q,l}^{(i)}}
\newcommand{\Bui}{B_{Q,1}^{(i)}}
\newcommand{\Bki}{B_{Q,k}^{(i)}}
\newcommand{\Bli}{B_{Q,l}^{(i)}}
\newcommand{\fun}{\frac{z^i}{(z-r_1)^k(z-r_2)^k}dz}

\newtheorem{definition}{Definition}[section]
\newtheorem{lemma}{Lemma}[section]
\theoremstyle{definition}
\newtheorem{remark}{Remark}[section]

\include{thebibliography}
\begin{document}

\title[\resizebox{5.5in}{!}{A Shimura-Shintani correspondence for rigid analytic cocycles of higher weight }]{A Shimura-Shintani correspondence for rigid analytic cocycles of higher weight }

\author{Isabella Negrini}

\begin{abstract}

This paper takes the first steps towards a systematic study of \emph{additive} rigid meromorphic cocycles of higher weight. These were introduced by Darmon and Vonk, who focused on multiplicative and weight two cocycles. After classifying certain rigid meromorphic cocycles of weight $2k$, we construct an explicit
 holomorphic kernel function realising 
a   Shimura-Shintani style correspondence from modular forms of
weight $k+1/2$ and level $4p^2$ to
rigid analytic cocycles of weight $2k$ 
on $\SL_2(\Z[1/p])$.
 \end{abstract}

\address{I.N.: Montreal, Canada}
\email{isabella.negrini@mail.mcgill.ca }

\maketitle

\tableofcontents

\section*{Introduction}

\noindent
Let $D>0$ be a real quadratic discriminant, and  let
\begin{equation}
\label{eqn:def-fkD}
 f_{k,D}(z) := \sum_{{\rm disc}(Q)=D} Q(z,1)^{-k}, \qquad 
k> 2  \mbox{ even}, \quad z\in  \C,  \ \ {\rm Im}(z)>0,
\end{equation}
where the sum runs over the integral binary quadratic forms $Q(z,1)=az^2+bz+c$ of discriminant $D$.  
This function,  which
   was first  considered in 
\cite[Appendix 2]{zagier-mfrq},
 is a weight $2k$ cusp form on $\SL_2(\Z)$, i.e. an element of $S_{2k}(\SL_2(\Z))$. 
 In 
\cite{KZ1}, it is shown
 to be the $D$-th Fourier coefficient  of the
holomorphic kernel function realising
 the Shimura-Shintani correspondence $\mathcal{S}$ from 
  the  ``Kohnen plus space"  $S_{k+1/2}^+$ of 
 cusp forms of weight $k+\frac{1}{2}$ on $\Gamma_0(4)$ having a Fourier development of the form 
 $$  g(z) = \sum_{n\geq 1} c(n) q^n, \mbox{ with } 
 \qquad c(n)=0 \mbox{ unless } n \equiv 0 \mbox{ or } 1 
 \pmod{4}.$$ 
More precisely, Theorem 2 of loc.cit. 
asserts that for each fixed $z$ in the usual upper half-plane $\mathcal{H}$, the generating series
\begin{equation}
\label{eqn:def-Omegaztau}
 \Omega_k(z,\tau) := \sum_{D>0} D^{k-1/2} f_{k,D}(z) e^{2\pi i D\tau}
 \end{equation}
belongs to  $S_{k+1/2}^+$ as a function of $\tau\in \mathcal{H}$. To any $g\in S_{k+1/2}^+$, the correspondence $\mathcal{S}$ associates an element of $S_{2k}(\SL_2(\Z))$ which, up to a multiplicative constant, is given by
$$
 \mathcal{S}(g)(z)=\frac{1}{6}\int_{\Gamma_0(4)\backslash\mathcal{H}} g(\tau)\overline{\Omega_k(-\overline{z},\tau)}v^{k-3/2}dudv.
$$

Let $p$ be a prime number and let $k\geq 3$ be an \emph{odd} integer. The goal of this paper is to exhibit an analogous kernel function $\hat{\Omega}_k$ of weight $k+1/2$ and level $4p^2$, 
 in which the space $S_{2k}(\SL_2(\Z))$ is replaced by
the space of {\em rigid analytic cocycles} of weight $2k$ on 
Ihara's group $\Gamma:= \SL_2(\Z[1/p])$ introduced in
\cite{DV1}. 
The function $\hat{\Omega}_k$ gives rise to a correspondence $\mathcal{C}$ from the space $S_{k+1/2}^{(\bar{\Q})}(\Gamma(4p^2))$ of weight $k+1/2$ cusp forms of level $4p^2$ with Fourier coefficients in $\bar{\Q}$ to the space of weight $2k$ rigid analytic cocycles.

One of the  themes of \cite{DV1} is that rigid analytic cocycles enjoy
a  strong parallel  with classical modular forms, 
while lending themselves to  complementary  applications, 
 notably to the analytic construction of class fields of real quadratic fields via their ``values" at real multiplication points of Drinfeld's $p$-adic upper half-plane.
 The    
 counterparts for rigid analytic cocycles  of \eqref{eqn:def-fkD} 
and \eqref{eqn:def-Omegaztau}   fit  into program of  developing the analogy
 between modular forms and rigid analytic cocycles initiated in \cite{DV1}.

This work also fits in the nascent “$p$-adic Kudla program”, which explores connections between automorphic forms and generating series of  cycles constructed by $p$-adic analytic means. As an example, in \cite{DV2} Darmon and Vonk relate Fourier coefficients of certain weakly holomorphic modular forms to divisors of rigid meromorphic cocycles, which can
then be viewed as real quadratic counterparts of Borcherds’ singular theta lifts. Darmon and Vonk start with a weakly holomorphic modular form $\psi$ of weight $1/2$ on $\Gamma_0(4p)$ and associate to it a rigid meromorphic cocycle whose singularities are concentrated at real multiplication (RM) points on Drinfeld's $p$-adic upper half-plane $\Hp=\mathbb{P}_1(\C_p)-\mathbb{P}_1(\Q_p)$. The singularities of this rigid meromorphic cocycle are determined by the principal part of $\psi$ and the result of \cite{DV2} adds evidence in favor of the analogy between rigid meromorphic cocycles and meromorphic functions whose divisors are concentrated at CM points, such as those arising in the image of Borcherds’ lift. Indeed, in \cite{B1} Borcherds associated to a weight $1/2$ weakly holomorphic modular form $\phi$ a certain $\SL_2(\Z)$-invariant real analytic function with logarithmic singualrities concentrated at CM points in $\mathcal{H}$ and determined by the principal part of $\phi$.

Similar correspondences have been studied by Bruinier and Ono (\cite{BrO}), Schwagenscheidt (\cite{Schw}), Oda (\cite{Oda}), and many others. These correspondences are usually defined via some theta kernel, and we will do the same by defining $\hat{\Omega}_k$.

\vspace{2mm}

\subsection*{Statement of results and outline}

Recall that $\Gamma:= \SL_2(\Z[1/p])$ is the Ihara group. Let $\mathcal{A}_k$ (resp. $\mathcal{M}_k$) be the additive group of rigid analytic (resp. meromorphic) functions on 
$\Hp$, endowed with the ``weight $k$ action" of $\Gamma$ given by
$$ h|\gamma (z) = (cz+d)^{-k} h\left(\frac{az+b}{cz+d}\right), 
\quad \mbox{ for } \gamma = \begin{pmatrix}a & b \\c & d \end{pmatrix}\in \Gamma.$$
Precise definitions of rigid analytic and meromorphic functions can be found in \cite{Das} (Sections 1 and 2) and \cite{gerritzen-vdp} (Chapter 2). For the purpose of this paper, a rigid analytic cocycle of weight $k$ 
 is a function 
$$ J: \mathbb{P}_1(\Q) \times \mathbb{P}_1(\Q) \rightarrow  \mathcal{A}_k $$ 
satisfying the ``modular symbol properties"
$$
 J\{r,s\}=- J\{s,r\}\:\:\:\: \text{ and }\:\:\:\: J\{r,s\} + J\{s,t\} = J\{r,t\}, \qquad \mbox{ for all } r,s,t\in \mathbb{P}_1(\Q), 
$$
together with the $\Gamma$-invariance condition
$$ J\{\gamma r, \gamma s\}|\gamma  = J\{r,s\}, \mbox{ for all } \gamma\in \Gamma = \SL_2(\Z[1/p]).$$
In other words, a rigid analytic cocycle is an element of $\MS^{\Gamma}(\mathcal{A}_k)$, the space of $\Gamma$-invariant modular symbols with values in $\mathcal{A}_k$. Similarly, rigid meromorphic cocycles are elements of $\MS^{\Gamma}(\mathcal{M}_k)$. This definition is equivalent to the one given in \cite{DV1} using parabolic cohomology (see \cite{DV1}, Corollary 1.10, for a proof of this fact). 

In Section \ref{classification meromorphic} we will classify rigid meromorphic cocycles of even weight satisfying a certain condition.

Given $D>0$ as above, let $\mathcal{F}_D(\Z[1/p])$ denote
 the set of binary quadratic forms of discriminant $D$ with coefficients in $\Z[1/p]$, equipped with its natural action of $\Gamma$.
Given a quadratic form   $Q(x,y) = ax^2+bxy + cy^2 \in \mathcal{F}_D(\Z[1/p])$, let 
$r_1$ and $r_2$ denote the so-called {\em first} and {\em second roots}
of $Q(z,1)$, defined by
\begin{equation}
 \label{fs roots}
r_1 = \frac{-b+\sqrt{D}}{2a}, \qquad r_2 = \frac{-b-\sqrt{D}}{2a},
\end{equation}
where $\sqrt{D}$ denotes the positive square root of $D$. 
Let  $\gamma_Q := (r_1,r_2)$ denote the hyperbolic geodesic 
going from $r_1$ to $r_2$, and for any pair $(r,s) \in \mathbb{P}_1(\Q)\times\mathbb{P}_1(\Q)$,
let $(r,s)$ likewise denote the hyperbolic geodesic joining $r$ to $s$ on 
$\mathcal{H}$. The choice of an orientation on $\mathcal{H}$ (following the usual 
``right hand rule" for instance)
determines an  intersection pairing between 
 hyperbolic geodesics, which is denoted $\gamma_1\cdot \gamma_2$, and   belongs to $\{-1,0,1\}$. 
 
In Section \ref{pf thm A} we will define a counterpart of Zagier's form
 $f_{k,D}(z)$ of 
 \eqref{eqn:def-fkD}  in   the setting of
 rigid analytic cocycles via the following theorem:
 
 \vspace{4mm}

\noindent
{\bf Theorem.}
{\em 
 Let $k\ge 1$ be odd. 
 For all $(r,s)\in \mathbb{P}_1(\Q)\times \mathbb{P}_1(\Q)$, the infinite sum
 $$ J_{k,D}\{r,s\}(z) := \sum_{Q\in \mathcal{F}_D(\Z[1/p])} (\gamma_Q \cdot (r,s)) \cdot Q(z,1)^{-k}$$
 converges to a rigid meromorphic function of $z\in \mathcal{H}_p$, which is rigid analytic when $(\frac{D}{p})=1$.  
 The function $$J_{k,D}: \mathbb{P}_1(\Q)\times \mathbb{P}_1(\Q) \rightarrow \mathcal{M}_{2k}$$
 is a rigid meromorphic cocycle  of weight $2k$ for $\SL_2(\Z[1/p])$. 
}

\vspace{4mm}
 
Fix a set of representatives $\mathbb{F}^{+}_p$ for $\mathbb{F}^{\times}_p/\{\pm 1\}$. For any $D$ with $(\frac{D}{p})=1$, we will denote by $\sqrt{D}\in\Q_p$ the root of $D$ in $\Q_p$ which is congruent to an element of $\mathbb{F}^{+}_p$ modulo $p$. Our main theorem is:

\vspace{4mm}

\noindent
{\bf Theorem.}
{\em
Let $k\ge 3$ be odd. If $D$ is not a square and $\big (\frac{D}{p}\big)= 1$, then $D^{k-1/2}J_{k,D}$ is the $D$-th coefficient of a weight $k+1/2$ cusp form $\hat{\Omega}_k(q)$ of level $4p^2$ with coefficients in $\MS^{\Gamma}(\mathcal{A}_{2k})$. The $D$-th coefficient of $\hat{\Omega}_k(q)$ vanishes if $\big (\frac{D}{p}\big)\ne 1$. 
}

\vspace{4mm}

The proof of this will be completed in Section \ref{main juice} and the tools for it will be defined in the preceding sections. In Section \ref{period computation} we will define a level $p$ analogue $f_{k,D}^{(p)}\in S_{2k}(\Gamma_0(p))$ of the Zagier form $f_{k,D}$. The forms $f_{k,D}^{(p)}$ belong to a certain $\bar{\Q}$-subspace  $\mathfrak{S}_{2k}^{(p)}(\bar{\Q})$ of $S_{2k}(\Gamma_0(p))$ and in Section \ref{main juice} we will package them into a generating series by:

\vspace{4mm}

\noindent
{\bf Theorem.}
{\em 
Let $k\ge 3$ be odd. Consider the series $\bar{\Omega}_k(q)=\sum_{D>0}D^{k-1/2}f_{k,D}^{(p)}\cdot q^D$, where $D$ ranges over discriminants with $(\frac{D}{p})=1$. Then $\bar{\Omega}_k$ is a weight $k+1/2$ cusp form of level $4p^2$ with coefficients in $\mathfrak{S}_{2k}^{(p)}(\bar{\Q})$.

}

In Section \ref{period computation} we will compute the period polynomials of $f_{k,D}^{(p)}$, getting a result analogous to Theorem 4 of \cite{KZ2}. We will also define a Schneider-Teitelbaum lift in the setting of rigid analytic cocycles of higher weight in Section \ref{ST}. The classical Schneider-Teitelbaum lift already appeared in \cite{Sch} and \cite{Tei}, and was extended to rigid analytic cocycles of weight $2$ in \cite{DV2}. We extended to higher weight the construction of this map and of its left inverse.

\vspace{2mm}

\subsection*{Some notation}

Given a binary quadratic form $Q(x,y)=ax^2+bxy+cy^2$, we will often adopt the notation $Q=[a,b,c]$, and throughout the paper we will use the notion of first and second root of $Q$ given in (\ref{fs roots}). The notions of the geodesic $\gamma_Q$ and the intersection number $\gamma_Q\cdot(r,s)$ given above will also be consistent throughout the paper.  A binary quadratic form $[a,b,c]$ with positive discriminant will be called \emph{simple} if $ac<0$, which implies that the two roots have opposite sign. A form $[a,b,c]$ such that $gcd(a,b,c)=1$ will be called \emph{primitive}.

For a positive discriminant $D$, we will denote by $\sqrt{D}=D^{1/2}$ its positive square root in $\mathbb{R}$ or a fixed square root in $\Q_p$, if $(\frac{D}{p})=1$. It will be clear from the context whether we are taking the real or the $p$-adic root. In order to choose $\sqrt{D}\in\Q_p$, we will fix a set of representatives $\mathbb{F}^{+}_p$ for $\mathbb{F}^{\times}_p/\{\pm 1\}$ and denote by $\sqrt{D}\in\Q_p$ the root of $D$ in $\Q_p$ which is congruent to an element of $\mathbb{F}^{+}_p$ modulo $p$.

The notion of modular symbol is used heavily in this paper, so we give its precise definition below.

\vspace{2mm}

\noindent
{\bf Definition.}
{\em 
Let $H$ be a subgroup of $\SL_2(\Q)$ and let $\Omega$ be a module over $H$, where we denote the group action by $\omega|h$ for $\omega \in \Omega$ and $h \in H$. A modular symbol with values in $\Omega$ is a function $m:\mathbb{P}_1(\Q)\times \mathbb{P}_1(\Q) \rightarrow \Omega$ such that
$$
m\{r,s\}=- m\{s,r\}\:\:\:\: \text{ and }\:\:\:\: m\{r,s\} + m\{s,t\} = m\{r,t\}, \qquad \mbox{ for all } r,s,t\in \mathbb{P}_1(\Q).
$$
A modular symbol $m$ is said to be $H$-invariant if
$$
m\{h r, h s\}|\gamma  = m\{r,s\}, \qquad \mbox{ for all } h\in H.
$$
}
We will use certain concepts from rigid analytic geometry such as the Bruhat-Tits tree $\T$ of PGL$_2(\Q_p)$ and the reduction map $\mathcal{H}_p \rightarrow \mathcal{T}$. Some references that cover this material are \cite{Das} (Section 1 and 2) and \cite{gerritzen-vdp} (Chapter 1 and 2). We will now fix some notation about these concepts.
We will denote by $v_0$ the standard vertex of $\T$, which is the vertex associated to the lattice $\Z_p^2$. Let $\T^{\leq n}$ be the subgraph of $\T$ containing all the vertices at distance at most $n$ from $v_0$, as well as all the edges containing two such vertices. The affinoid subdomain of $\Hp$ given by points reducing to $\T^{\leq n}$ will be denoted by $\Hp^{\leq n}$. This affinoid subdomain is obtained by removing $(p+1)p^n$ open disks of radius $p^{-n}$ from $\PCp$. Similarly, we will denote by $\T^{< n}$ the subgraph of $\T$ containing all vertices at distance at most $n-1$ from $v_0$, as well as all the edges containing at least one of these vertices. The wide open subspace of $\Hp$ made of all the points reducing to $\T^{<n}$ will be denoted by $\Hp^{<n}$. Let $\T_0$ be the set of vertices of $\T$, let $\T_1$ be the set of edges and let $\T_1^*$ be the set of ordered edges of $\T$. A vertex is said to be even (resp. odd) if it has an even (resp. odd) distance from $v_0$. An ordered edge is said to have an even (resp. odd) orientation if its source is an even (resp. odd) vertex. We will denote by $\T_1^+$ the set of edges which have an even orienation and by $\T_1^-$ the set of edges which have an odd orientation. The standard edge of $\T$ with positive orientation will be denoted by $e_0$. For any $e \in \T_1^*$, we will denote by $\bar{e}$ the same edge taken with the opposite orientation and by $s(e)$ the \emph{source} of $e$, i.e. the vertex where $e$ starts.

\vspace{4mm}

\section{Classification of certain rigid meromorphic cocycles of weight $2k$}
\label{classification meromorphic}

In this section we classify rigid \emph{meromorphic} cocycles of weight $2k$ satisfying a certain condition on their poles and residues. This is heavily inspired by the classification of rigid meromorphic cocycles of weight two carried out in \cite{DV1} and most of the notation in this section is taken from there. The integer $k$ will be assumed to be odd in order to prove Theorem \ref{last thm}, but  the other results of this section hold for any $k \geq 0$. 

\begin{definition}
A rigid meromorphic period function of weight $k$ is the value at $\{0,\infty\}$ of a rigid meromorphic cocyle of weight $k$. The space of such functions will be denoted by $\mathcal{R}_k$.
\end{definition}
A pair $(a/b, c/d)\in \mathbb{P}_1(\Q)\times \mathbb{P}_1(\Q)$ is called unimodular if $ad-bc=\pm 1$.  Any pair in $\mathbb{P}_1(\Q)\times \mathbb{P}_1(\Q)$ can be decomposed as a sequence of unimodular pairs and $\Gamma$ acts transitively on such pairs. This implies that, in order to give a classification for $\MS^{\Gamma}(\mathcal{M}_k)$, it is enough to classify $\mathcal{R}_k$. A function $\varphi\in\mathcal{R}_k$ satisfies the following identities:
 \begin{equation}
\label{2 3 rel}
    \varphi|(1+S)=0,\:\:\:\varphi|(1+U+U^2)=0, \:\:\:\varphi|D=\varphi,
    \end{equation}
    where
    $$
    S=\begin{pmatrix}0 &1\\ -1 & 0  \end{pmatrix},\text{  }\text{  }\text{  }\text{  }U=\begin{pmatrix}0 &1\\ -1 & 1  \end{pmatrix},\text{ }\text{ }\text{ }\text{ }D=\begin{pmatrix}p &0\\ 0 & 1/p  \end{pmatrix}.
    $$
This follows from the modular symbol properties. Moreover, the matrix
$$
P:=\begin{pmatrix}p &0\\ 0 & 1  \end{pmatrix}
$$
induces an involution $\varpi_p$ on $\mathcal{R}_k$, defined by
$$
\varpi_p(\varphi)(z):=-\varphi|P(z)=-p^{k/2}\varphi(pz).
$$
A rigid meromorphic period function is said to be $p$-even (resp. $p$-odd) if it satisfies
$$
\varpi_p(\varphi)=\varphi, \:\:\:\:(\text{resp }\varpi_p(\varphi)=-\varphi).
$$
Note that $P^2=pD$.
\begin{definition}
A point $\tau\in\Hp$ is said to be a real multiplication (RM) point if $\Q(\tau)$ is a real quadratic field. The set of such points is be denoted by $\Hp^{RM}$. The Galois conjugate of $\tau\in\Hp^{RM}$ is denoted by $\tau^{\prime}$.
\end{definition}

Fix an embedding of the real quadratic field $\Q(\tau)$ into $\mathbb{R}$. We will denote by $\sqrt{D}$ the square root of $D$ in $\Q_p$ corresponding via the fixed embedding to the positive real root of $D$.

\begin{definition}
Let $\omega \in\Hp^{RM}$ and let $r,s\in \mathbb{P}_1(\Q)\times \mathbb{P}_1(\Q)$. Let $(r,s)$ be the hyperbolic geodesic joining $r$ and $s$ in the upper-half plane $\mathcal{H}$. Let $\gamma_{\omega}$ be  the hyperbolic geodesic joining $\omega$ and $\omega^{\prime}$. We say that $\omega$ is linked to $(r,s)$ if $\gamma_{\omega}$ intersects $(r,s)$.
\end{definition}

Let $\tau \in\Hp^{RM}$ and fix an embedding of $\Q(\tau)$ into $\R$. Let
$$
\Sigma_{\tau}(0,\infty):=\{\omega \in \Gamma\cdot \tau\: \text{ such that }\: \omega\omega^{\prime}<0\}.
$$
In more generality, one can define
$$
\Sigma_{\tau}(r,s):=\{\omega \in \Gamma\cdot \tau\: \text{ such that }\: \omega \:\text{ is linked to }(r,s)\}.
$$
The set $\Sigma_{\tau}(r,s)$ is endowed with a sign function $\delta_{r,s}:\Sigma_{\tau}(r,s)\rightarrow \pm 1$. The value $\delta_{r,s}$ depends on whether $\omega$ is \enquote{inside} or \enquote{outside} the semicircle $(r,s)$ (for more details, see Equation (19) of \cite{DV1}). Note that $\delta_{0,\infty}(\omega)$ coincides with the sign of $\omega$ in $\R$. The intersection of $\Sigma_{\tau}(r,s)$ with any affinoid in $\Hp$ is finite.

\begin{lemma}
\label{magic lemma}
Let $Q=[A,B,C]$ be a primitive binary quadratic form  with positive discriminant $D=D_0p^m$, where $p\nmid D_0$ and $m>0$. Let $z$ be a point in the affinoid $\Hp^{\leq n}$ for some fixed $n$. If $n<m/2$, then
$$
\Big |  \frac{1}{Q(z,1)^k} \Big |<p^{2nk}.
$$
\end{lemma}
\begin{proof}
Note that, if $p$ does not split in $\Q(\sqrt{D})$, then the roots $r_1, r_2$ of $Q$ reduce to a point of $\T$ at distance $m/2$ from the standard vertex (see for example Proposition 1.1 of \cite{DV1}), therefore $|z-r_i|>1/p^n$ if $m/2>n$. If $p$ splits in $\Q(\sqrt{D})$, the inequality $|z-r_i|>1/p^n$ still holds, because of equations (1.2.4) of \cite{Das}. If $p\nmid A$, the lemma follows immediately as 
$$
\frac{1}{Q(z,1)^k}=\frac{1}{(Az^2+Bz+C)^k}=\frac{1}{A^k(z-r_1)^k(z-r_2)^k}.
$$
Now assume that $p|A$. Note that $p$ must divide $B$, as $m>0$. This implies that $p\nmid C$, otherwise $Q$ would not be primitive.  Assume $A>0$ (if $A<0$, we proceed similarly), so that $A=\frac{D_0p^{m}-B^2}{4|C|}$ and the two roots $r_1$ and $r_2$ of $[A,B,C]$ can be written as 
$$
\frac{-B\pm p^{m/2}\sqrt{D_0}}{2A}=\frac{2|C|(-B\pm p^{m/2}\sqrt{D_0})}{D_0p^{m}-B^2}.
$$
Hence we get 
\begin{equation}
\label{term!}
\frac{1}{A^k(z-r_1)^k(z-r_2)^k}=\frac{4|C|^k}{(z(p^{m/2}\sqrt{D_0}+B)-2|C|)^k(z(p^{m/2}\sqrt{D_0}-B)+2|C|)^k}.
\end{equation}
We can write $val_p(z-2|C|/B) \leq n-2val_p(B)$ (this is one of the equations defining $\Hp^{\leq n}$, see for example equation (1.2.4) of \cite{Das}). Hence we have $val_p(zB-2|C|)\leq n-val_p(B)\leq n$, and
so the norm of (\ref{term!}) is smaller than $p^{2nk}$ for $n<m/2$.
\end{proof}

To any $\omega \in\Hp^{RM}$ we can associate a binary quadratic form $Q_{\omega}$ with integer coprime coefficients such that $Q_{\omega}(\omega,1)=0$ and $\omega$ is the first root of $Q_{\omega}$. The correspondence $\omega \leftrightarrow Q_{\omega}$ is bijective. The discriminant of $Q_{\omega}$ is the discriminant of $\omega$, which is a positive integer $D_{\omega}=p^{n_{\omega}}D_{\omega,0}$ where $p$ does not divide $D_{\omega,0}$.

\begin{theorem}
\label{Thm conv}
For any $\tau \in \Gamma\setminus\Hp^{RM}$, the infinite sum
$$
\varphi_{\tau}(z):=\sum_{\omega\in\Sigma_{\tau}(0,\infty)}(\gamma_{Q_{\omega}}\cdot(0, \infty))\frac{D_{\omega}^{k/2}}{Q_{\omega}(z,1)^k}.
$$
converges uniformly on affinoid subsets to a rigid meromorphic period function of weight $2k$.
\end{theorem}
\begin{proof}
We will assume that $\tau$ reduces to a vertex of the Bruhat-Tits tree $\T$, as the case in which $\tau$ reduces to an edge can be treated in a similar way. We will show that the restriction of $\varphi_{\tau}$ to any affinoid $\Hp^{\leq n}$ can be written as the limit for $h \to \infty$ of a Cauchy sequence $\{\varphi^{(h)}_{\tau}\}$ of rational functions and that it has finitely many poles in $\Hp^{\leq n}$. Indeed, Proposition 1.1 of \cite{DV1} implies that $\varphi_{\tau}$ is the limit for $h \to \infty$ of the rational functions
$$
\varphi^{(h)}_{\tau}(z):=\sum_{\omega\in\Sigma_{\tau}^{\leq h}}(\gamma_{Q_{\omega}}\cdot(0, \infty))\frac{D_{\omega}^{k/2}}{Q_{\omega}(z,1)^k}, \:\:\:\:\text{ where }\:\:\:\:\Sigma_{\tau}^{\leq h}:=\Sigma_{\tau}(0,\infty) \cap \Hp^{\leq h},
$$
as $p^{n_{\omega}}$ grows for $\omega\in\Sigma_{\tau}(0,\infty)-\Sigma_{\tau}^{\leq h}$, when $h$ grows.
If $\omega \in \Sigma_{\tau}^{\leq h}$, then $\omega$ reduces to a vertex of $\T$ at distance $N\leq h$ from the standard vertex $v_0$. If $N>n$, then $\omega$ does not belong to $\Hp^{\leq n}$ and $Q_{\omega}(z,1)^{-k}$ is regular on $\Hp^{\leq n}$.

For $z\in \Hp^{\leq n}$ with $n$ fixed and $h$ big enough, Lemma \ref{magic lemma} implies that
$$
\Big| \frac{D_{\omega}^{k/2}}{Q_{\omega}(z,1)^k}    \Big|_p<p^{k(2n-n_{\omega}/2)}.
$$
Hence $\{\varphi^{(h)}_{\tau}\}_h$ is a Cauchy sequence. We have proven that $\varphi_{\tau}(z)$ converges to a rigid meromorphic function on $\Hp$. Similarly the sums
$$
\Phi_{\tau}\{r,s\}:=\sum_{\omega\in\Sigma_{\tau}(r,s)}(\gamma_{Q_{\omega}}\cdot(r, s))\frac{D_{\omega}^{k/2}}{Q_{\omega}(z,1)^k}
$$
converge to rigid meromorphic functions. We will now prove that
$$
\Phi_{\tau}: \mathbb{P}_1(\Q)\times \mathbb{P}_1(\Q)\rightarrow \mathcal{M}_{2k}
$$
is a $\Gamma$-invariant modular symbol. The modular symbol property holds because any pair $(\omega, \omega^{\prime})$ is linked to the pair $(r,t)$ if and only if it is linked either to $(r,s)$ or $(s,t)$ but not to both. (If it is linked to both, then the two intersection numbers have opposite signs).

Now we show the $\Gamma$-invariance condition, i.e. $\Phi_{\tau}\{\delta^{-1} r,\delta^{-1} s\}=\Phi_{\tau}\{r,s\}|\delta$ for any $\delta \in \Gamma$, where the action of $\Gamma$ on $\mathcal{M}_{2k}$ is the weight $2k$ one. Given a form $Q=[A,B,C]$ with roots $\omega,\omega^{\prime}$ and a matrix $\delta= \begin{pmatrix} a & b \\ c& d \end{pmatrix}$, we compute
$$
Q(z,1)^{-k}|\delta =\Big(\frac{1}{A(z-\omega)^k(z-\omega^{\prime})^k}\Big)\big|\delta =\frac{(a-\omega c)^{-k}(a-\omega^{\prime}c)^{-k}}{A(z-\delta^{-1}\omega)^k(z-\delta^{-1}\omega^{\prime})^k} .
$$
Letting $Q|\delta=[A^\prime,B^\prime, C^\prime]$ where the action of $\Gamma$ on quadratic forms is the usual one, the expression above can be rewritten as
$$
\frac{1}{A^\prime(z-\delta^{-1}\omega)^k(z-\delta^{-1}\omega^{\prime})^k}= (Q|\delta )(z,1)^{-k},
$$
hence
$$
\Phi_{\tau}\{r,s\}(z)|\delta=\sum_{\omega\in\Sigma_{\tau}(r,s)} (\gamma_{Q_{\omega}}\cdot(r, s))\frac{D_{\omega}^{k/2}}{(Q_{\omega}|\delta)(z,1)^k}.
$$
Now note that $ (\gamma_{Q_{\omega}}\cdot(\delta^{-1}r,\delta^{-1}s))= (\gamma_{Q_{\omega}|\delta^{-1}}\cdot(r,s))$, so
$$
\Phi_{\tau}\{\delta^{-1} r,\delta^{-1} s\}(z)=\sum_{\omega\in\Sigma_{\tau}(r,s)}(\gamma_{Q_{\omega}|\delta^{-1}}\cdot(r, s))\frac{D_{\omega}^{k/2}}{Q_{\omega}(z,1)^k},
$$
and the $\Gamma$-invariance follows.
The theorem follows from the fact $\varphi_{\tau}=\Phi_{\tau}\{0,\infty\}$.
\end{proof}

Similarly to \cite{DV1}, we provided an explicit collection of rigid meromorphic period functions $\varphi_{\tau}$ of weight $2k$ indexed by $\tau \in \Gamma\setminus\Hp^{RM}$. However in our case the sets of poles of these functions are  given by $\Sigma_{\tau}(0,\infty)\cup \Sigma_{\tau^{\prime}}(0,\infty)$, and not only by $\Sigma_{\tau}(0,\infty)$. For this reason, the rigid meromorphic period functions $\phi$ that we classify have to satisfy the condition that if $\tau\in\Hp^{RM}$ is a pole then also $\tau^{\prime}$ is a pole, and $res_{\tau}(\phi)=-res_{\tau^{\prime}}(\phi)$. 

Note that $\varphi_{\tau}=\varphi_{\tau^{\prime}}$, since $Q_{\omega}=-Q_{\omega^{\prime}}$ and $(\gamma_{Q_{\omega}}\cdot(0, \infty))=-(\gamma_{Q_{\omega^{\prime}}}\cdot(0, \infty))$. However, if we consider the index $\tau \in \Gamma\setminus\Hp^{RM}$ modulo the Galois action, then the functions $\varphi_{\tau}$ are linearly independent, as they have disjoint set of poles.

\begin{theorem}
\label{last thm}
Let $k\geq 1$ be odd. Let $\phi$ be a rational period function of weight $2k$, and assume that if a point $\tau \in\Hp^{RM}$ is a pole of $\phi$, then also $\tau^{\prime}$ is a pole of $\phi$. Assume moreover that $res_{\tau}(\phi)=-res_{\tau^{\prime}}(\phi)$. Then $\phi$ is a finite linear combination of the functions $\varphi_{\tau}$ of Theorem \ref{Thm conv} and of a rigid analytic period function of weight $2k$.
\end{theorem}
\begin{proof}
This proof parallels closely the proof of Theorem 1.24 of \cite{DV1}. 

Let $\phi^{\pm}:=\phi\pm\varpi_p(\phi)$. As $\phi=(\phi^++\phi^-)/2$, we can assume without loss of generality that $\phi$ is $p$-even or $p$-odd. Let $\Sigma_{\phi}$ denote the set of poles of $\phi$. The identities in (\ref{2 3 rel}) imply 
\begin{equation}
\label{poles implications}
\omega \in \Sigma_{\phi} \Rightarrow S(\omega)\in \Sigma_{\phi} \:\:\:\:\text{ and }\:\:\:\: U(\omega)\in \Sigma_{\phi}\: \text{ or }\: U^2(\omega) \in\Sigma_{\phi}.
\end{equation}
Now let
$$
\Sigma_{\phi}^{<1}:=\Sigma_{\phi}\cap \Hp^{<1}.
$$
The set $\Sigma_{\phi}^{<1}$ is finite because a rigid meromorphic function has finitely many poles on any given affinoid. Moreover, the fact that SL$_2(\Z)$ preserves $\Hp^{<1}$ implies that $\Sigma_{\phi}^{<1}$ satisfies the conditions in (\ref{poles implications}) just as $\Sigma_{\phi}$ does. Any finite set satisfying (\ref{poles implications}) can be written as a finite union
$$
\Sigma_{\phi}^{<1}=\cup_{\tau\in I_{\phi}}\Sigma^0_{\tau}(0,\infty), 
$$
where
$$
I_{\phi}\subset \SL_2(\Z)\setminus(  \Hp^{RM}\cap\Hp^{<1})\:\:\:\:\text{ and }\:\:\:\:\Sigma^0_{\tau}(0,\infty)=\Sigma_{\tau}(0,\infty)\cap\Hp^{<1}.
$$
See \cite{DV1} for a proof of this fact. Lemma 4 and Lemma 5 of \cite{CZ} imply that $\phi$ has only poles of order $k$ on $\Hp^{<1}$ and that if $\tau\in\Hp^{<1}$ is a pole of $\phi$, then
\begin{equation}
\label{formula PP}
\text{PP}_{\tau}(\phi)=res_{\tau}((z-\tau)^{k-1}\phi(\tau))\cdot\text{PP}_{\tau}(D^{k/2}/Q_{\tau}(z,1)^k),
\end{equation}
where PP$_{\tau}$ denotes the principal part of a function at $\tau$. Lemma 4 of \cite{CZ} also implies that, for any $\SL_2(\Z)$-orbit $\mathcal{A}$ of a pole $\tau\in\Hp^{<1}$ of $\phi$, there exist a constant $C_{\mathcal{A}}$ such that
\begin{equation}
\label{formula residue}
res_{\tau}((z-\tau)^{k-1}\phi(\tau))=\text{sgn}(\tau)\cdot C_{\mathcal{A}}.
\end{equation}
Let $\mathcal{A}_{\tau}$ be the $\SL_2(\Z)$-orbit of a pole $\tau$. As $k$ is odd, equations (\ref{formula PP}) and (\ref{formula residue}), together with the conditions on the poles and residues of $\phi$ in the statement of the Theorem, imply that
$$
\phi-\sum_{\tau\in I_{\phi}}C_{\mathcal{A}_{\tau}}\varphi_{\tau}^+
$$
is a $p$-even rigid meromorphic period function having no singularities outside $\Hp^{<1}$. Any $p$-even or $p$-odd rigid meromorphic period function which is regular on $\Hp^{<1}$ must be regular on $\Hp$. This is proven in \cite{DV1} for functions of weight $2$, but the proof is identical for functions of higher weight.
\end{proof}

\vspace{4mm}


\section{Rigid analytic cocycles attached to real quadratic fields}
\label{pf thm A}

In this section, we associate to any real quadratic discriminant $D$ satisfying $(\frac{D}{p})=1$ (resp. $(\frac{D}{p})\ne1$) a rigid analytic (resp. meromorphic) cocycle $J_{k,D}$ of weight $2k$ for the Ihara group. The cocycle $J_{k,D}$ is one of the main objects studied in this paper, as it is a $p$-adic analogue of the Zagier form $f_{k,D}$ mentioned in the introduction and it will be used in Section \ref{main juice} to construct the correspondence $\mathcal{C}$ that we seek.

\begin{theorem}
\label{def thm}
Let $k \geq 1$ be odd. For all $(r,s)\in \mathbb{P}_1(\Q)\times \mathbb{P}_1(\Q)$, the infinite sum
$$
J_{k,D}\{r,s\}(z):=\sum_{Q\in \mathcal{F}_D(\Z[1/p])}(\gamma_Q\cdot(r,s))\cdot Q(z,1)^{-k}
$$
converges to  a rigid meromorphic function of $z \in \Hp$, which is rigid analytic when $(\frac{D}{p})=1$. The function
$$
J_{k,D}\{r,s\}: \mathbb{P}^1(\Q)\times \mathbb{P}^1(\Q) \rightarrow \mathcal{M}_{2k}
$$
is a rigid meromorphic cocycle of weight $2k$ for $\SL_2(\Z[1/p])$.
\end{theorem}

\begin{proof}
For any form in $\mathcal{F}_D(\Z[1/p])$, we can clear the denominators, obtaining a form in $\mathcal{F}_{Dp^{2n}}(\Z)$ for some $n\in \Z$. Hence the sum in the statement can be rewritten as
\begin{equation}
\label{other expression}
\sum_{n=0}^\infty\Big(\sum_{Q\in \mathcal{F}_{Dp^{2n}}(\Z)} (\gamma_Q\cdot(r,s))\cdot Q(z,1)^{-k}\Big)p^{nk}.
\end{equation}
Let $z$ belong to the affinoid $\Hp^{\leq h}$ for some fixed $h$.
We assume at first that $(r,s)=(0,\infty)$, so the inner sum is over simple forms $[A,B,C]$ with $A,B,C \in \Z,$ $AC<0,$ thus $B^2+4|AC|=Dp^{2n}$. Note that, if $p$ does not split in $\Q(\sqrt{D})$, the roots $r_1, r_2$ of such a form reduce to a point of $\T$ at distance $n$ from the standard vertex (see for example Proposition 1.1 of \cite{DV1}), therefore $|z-r_i|>1/p^h$ if $n>h$, so the inner sum is regular on $\Hp^{\leq h}$. 

We are going to prove that, for any $h \geq 0$, the outer sum in (\ref{other expression}) is the limit of a Cauchy sequence relative to the sup norm on $\Hp^{\leq h}$ and hence converges to a rigid meromorphic function on $\Hp$. To do this, we will show that the norm of the general term of the outer sum is (eventually) going to zero uniformly in $z \in \Hp^{\leq h}$. Indeed, Lemma \ref{magic lemma} implies that eventually
\begin{equation}
\label{general term inner}
\Big |\frac{1}{Q(z,1)^k}\Big |<p^{2hk}.
\end{equation}

As there are only finitely many simple forms of a given discriminant, we get that 
$$
\sum_{Q\in \mathcal{F}_{Dp^{2n}}(\Z)} (\gamma_Q\cdot(0,\infty))\cdot Q(z,1)^{-k}
$$
is a finite sum and the general term of the outer sum in (\ref{other expression}) eventually has norm smaller than $p^{k(2h-n)}$, so the series converges to a rigid meromorphic function $J_{k,D}(0,\infty)(z)$, which is analytic when $(\frac{D}{p})=1$ (as in this case $\sqrt{D}\notin \Hp$).

The case of a general pair $(r,s)$ follows from the modular symbol property and the $\Gamma$-invariance condition, together with the fact that any pair $(r,s)$ can be written as a sum of unimodular pairs and SL$_2(\Z)$ acts transitively on such pairs. The modular symbol property and the $\Gamma$-invariance condition are proved with the same computations used at the end of the proof of Theorem \ref{Thm conv}.
\end{proof}

\vspace{4mm}

\section{A Schneider-Teitelbaum lift for rigid analytic cocycles}
\label{ST}
The aim of this section is to define a Schneider-Teitelbaum lift for rigid analytic cocycles, i.e. a map
$$
\ST:  \MS^{\Gamma_0(p)} (\mathcal{P}_{2k-2})\rightarrow \MS^\Gamma(\mathcal{A}_{2k}),
$$
where $\mathcal{P}_{2k-2}$ denotes polynomials with coefficients in $\C_p$ and degree at most $2k-2$, endowed with the following \enquote{weight $2k-2$ action} of $\Gamma$
\begin{equation}
\label{action on poly}
q|\gamma (z)=(cz+d)^{2k-2} q\Big(\frac{az+b}{cz+d}\Big), \:\:\:\: \text{for }\gamma = \begin{pmatrix} a &b \\ c& d \end{pmatrix}.
\end{equation}
For cocycles of weight $2$, this was already done in \cite{DV1}. It will be more convenient to consider the dual space $\mathcal{P}_{2k-2}^{\vee}:=\text{Hom}_{\C_p}(\mathcal{P}_{2k-2}, \C_p)$, which is a $\Gamma$-module with the action of $\Gamma$ given by
$$
(\hat{q}|\gamma) (\cdot)=\hat{q}(\:\cdot \:|\gamma^{-1}).
$$
The spaces $\mathcal{P}_{2k-2}$ and $\mathcal{P}_{2k-2}^{\vee}$ are isomorphic as $\Gamma$-modules, so we will actually define a map
$$
\ST^{\vee}:  \MS^{\Gamma_0(p)} (\mathcal{P}_{2k-2}^{\vee})\rightarrow \MS^\Gamma(\mathcal{A}_{2k})
$$
and at the beginning of Section \ref{res map} we will write down the corresponding map $\ST$ defined on $\MS^{\Gamma_0(p)} (\mathcal{P}_{2k-2})$.

\begin{definition}
A harmonic cocycle with value in a $\Gamma$-module $\Omega$  is a function $c: \T_1^*\rightarrow \Omega$ satisfying
$$
c(\bar{e})=-c(e), \:\:\text{and }\:\: \sum_{s(e)=v}c(e)=0, \:\:\:\text{ for all } \:\:\: v\in\T_0 \:\:\text{ and } \: \:e\in \T_1^*.
$$
The space of such harmonic cocycles is denoted by $C_{har}(\Omega)$. The actions of $\Gamma$ on $\T$ and $\Omega$ induce an action on $C_{har}(\Omega)$ given by
$$
(c|\gamma)(e)=c(\gamma e)|\gamma.
$$
\end{definition}

So we want to define a map MS$^{\Gamma_0(p)} (\mathcal{P}_{2k-2}^{\vee})\rightarrow $ MS$^\Gamma(\mathcal{A}_{2k})$, but because of the lemma below this is like defining a map  MS$^{\Gamma}( C_{\text{har}}(\mathcal{P}_{2k-2}^{\vee})) \rightarrow \text{MS}^{\Gamma}(\mathcal{A}_{2k})$.

\begin{lemma}
\label{Lemma: isom}
There is an isomorphism $\text{ev}_{e_0}: \MS^{\Gamma}(C_{\text{har}}(P^{\vee}_{2k-2})) \rightarrow  \MS^{\Gamma_0(p)}( P^{\vee}_{2k-2})$.
\end{lemma}
\begin{proof}
Let ev$_{e_0}$ be the $\Gamma_0(p)$-equivariant map induced by evaluating harmonic cocycles on the standard edge $e_0$ of the Bruhat-Tits tree $\T$.  We first prove its injectivity. If $c$ is a modular symbol in the kernel of ev$_{e_0}$, then $c \{r,s \}(e_0)=0$ for all $r,s \in \mathbb{P}_1(\Q)$. But $\Gamma$ acts transitively on $\T_1 ^+$ so for any edge $e$ we have $e=\gamma ^{-1} e_0$ for some $\gamma \in \Gamma$. The definition of the $\Gamma$-action on harmonic cocycles, together with the $\Gamma$-invariance of the modular symbol $c$, give:
$$
c \{ r,s \}(e)=c \{ r,s\}(\gamma ^{-1}e_0)= (c \{ r,s\}| \gamma ^{-1})  (e_0)|\gamma=( c \{\gamma r, \gamma s \}(e_0))|\gamma=0.
$$
We now prove the surjectivity. Given $c_0 \in \text{MS}^{\Gamma_0(p)}( P^{\vee}_{2k-2})$, define $c \in \text{MS}^{\Gamma}(C_{\text{har}}(P^{\vee}_{2k-2}))$ by setting, for each $e= \gamma ^{-1} e_0 \in \T ^+_1$,
$$
c \{r,s \}(e):= c_0\{ \gamma r,\gamma s \}| \gamma . 
$$
Note that $\gamma$ is only well-defined up to left multiplication by elements of Stab$_{\Gamma}(e_0)=\Gamma_0(p)$, but if we substitute $\gamma$ by $\delta\gamma$ with $\delta \in \Gamma_0(p)$, we get
$$
c_0\{\delta \gamma r, \delta \gamma s \}|\delta \gamma=(c_0\{\gamma r, \gamma s  \}| \delta^{-1})|\delta \gamma=c_0\{\gamma r, \gamma s   \}|\gamma.
$$
It is easy to see that $\text{ev}_{e_0}(c)=c_0$. Finally, for any $c \in \MS^{\Gamma}(C_{\text{har}}(P^{\vee}_{2k-2})) $ the modular symbol $\text{ev}_{e_0}(c)$ is $\Gamma_0(p)$-invariant, because $\Gamma_0(p)$ fixes $e_0$ and $c$ is $\Gamma$-invariant.
\end{proof}

\vspace{4mm}

To any oriented edge $e$ of the Bruhat-Tits tree $\T$, one can associate a $p$-adic ball $U(e)\subset \mathbb{P}_1(\mathbb{Q}_p)$. This is done in \cite{Das}, but we will briefly cover the construction below. We will need the notion of \emph{ends} on the tree $\T$. 

\begin{definition}
Let $P=(l_0, l_1,\dots)$ and $P'=(l^{\prime}_0, l^{\prime}_1,\dots)$ be infinite paths of vertices of $\T$ without backtracking. If $P$ and $P'$ differ only by a finite number of vertices, we say that they are equivalent and we write	$P\sim P'$.
An equivalence class $[P]$ for the relation $\sim$ is called an end of $\T$. The set of all ends is denoted by \emph{Ends}$(\T)$.
\end{definition}

Let $e$ be the oriented edge running from a vertex $l_0$ to a vertex $l_1$ and let 
$$	
U_e:=\Set{[P]\in \text{Ends}(\T)|P=(l_0,l_1\dots)},
$$
which is the subtree of $\T$ given by all ends leaving the oriented edge $e$. Let  $red: \:\mathcal{H}_p \rightarrow \mathcal{T}$ be the reduction map from $\Hp$ to $\T$ (for a precise definition see, for example, \cite{Das}). Let $\Sigma_e=red^{-1}(U_e)$ and let $\bar{\Sigma}_e$ be the closure of $\Sigma_e$ in $\PCp$, then 
\[
	U(e):=\bar{\Sigma}_e\cap\PQp
\]
is a ball in $\PQp$.  One can show that $U(\gamma e)=\gamma U(e)$ for any $\gamma \in \Gamma$ (see \cite{Das}).

Note that a modular symbol $c$ in $\text{MS}^{\Gamma}( C_{\text{har}}(\mathcal{P}_{2k-2}^{\vee}))$ is a collection of harmonic cocycles $c\{r, s\}$  indexed by $\mathbb{P}_1(\mathbb{Q}) \times \mathbb{P}_1(\mathbb{Q})$ and satisfying the modular symbol conditions. This gives a collection of distributions $\mu_{c  \{r, s\}}$ on $\mathbb{P}_1(\mathbb{Q}_p)$ defined by
\begin{equation}
\label{distribution dual}
\int_{U(e)}P(t)d\mu_{c\{r,s\}}(t)=(c\{r,s\}(e))(P),
\end{equation}
where $P \in \mathcal{P}_{2k-2}$. Note that $c\{r,s\}(e)$ is an element of $\mathcal{P}_{2k-2}^{\vee}$, hence it can be evaluated at $P\in\mathcal{P}_{2k-2}$. The distributions given in \ref{distribution dual} are basically the same as the ones defined in \cite{Sch} and \cite{Tei}. The only difference is that we are generalizing them to the setting of modular symbols.

The map  $\ST^{\vee}: \MS^{\Gamma}( C_{\text{har}}(\mathcal{P}_{2k-2}^{\vee})) \rightarrow \text{MS}^{\Gamma}(\mathcal{A}_{2k})$ will be given by $c \mapsto f$, where
\begin{equation}
\label{def Pois Ker}
f\{r,s\}(z)=\int_{\PQp} \frac{1}{z-t} d \mu_{c\{r,s\}}(t).
\end{equation}
We need to show that this expression makes sense, in particular we need to show that our integral extends to a set of functions containing $\frac{1}{z-t}$. This is done in Section \ref{The case of balls centered at zero}. We will then show in Section \ref{The end of the proof} that $f\{r,s\}(z)$ is an element of $\mathcal{A}_k$ and that $f\{r,s\}$ is a $\Gamma$-invariant modular symbol.

Note that the assignment  $c\{r,s\} \mapsto c\{0,\infty\}$ is injective and identifies  $\text{MS}^{\Gamma}( C_{\text{har}}(\mathcal{P}_{2k-2}^{\vee}))$  with the subset of $ C_{\text{har}}(\mathcal{P}_{2k-2}^{\vee})$ given by the  harmonic cocycles $c$ satisfying the relations
\begin{equation}
\label{conditions}
c|(1+S)=0, \text{  }\text{  }\text{  }\text{  }c|(1+U+U^2)=0, \text{  }\text{  }\text{  }\text{  } c|D=c,
\end{equation}
where
$$
S=\begin{pmatrix}0 &1\\ -1 & 0  \end{pmatrix},\text{  }\text{  }\text{  }\text{  }U=\begin{pmatrix}0 &1\\ -1 & 1  \end{pmatrix},\text{ }\text{ }\text{ }\text{ }D=\begin{pmatrix}p &0\\ 0 & 1/p  \end{pmatrix}.
$$
This means that, given $c \in C_{\text{har}}(\mathcal{P}_{2k-2}^{\vee})$ satisfying the relations above, it is enough to show that the integral
$$
\int_{\PQp} \frac{1}{z-t} d \mu_{c\{0,\infty\}}(t)
$$
makes sense. We will do this using the following fact, which can be found in \cite{Tei} and \cite{Ort}. Let $r$ be a fixed integer with $0 \leq r \leq k-2$ and  let $\mu$ be a distribution on $\mathcal{P}_{2k-2}$ satisfying 
\begin{equation}
\label{no inf}
\Big|\int_{U(e)}(x-a)^n d\mu(x)\Big|_p \leq C (1/p)^{\alpha(e)(n-r)} \text{ } \text{ }\text{ for } a \in U(e), \infty \notin U(e), 0 \leq n \leq 2k-2,
\end{equation}
where $\alpha (e)= \text{inf}_{u,v \in U(e)}\{ val_p(u-v) \}$ for $\infty \notin U(e)$, and 
\begin{equation}
\label{yes inf}
\Big|\int_{U(e)}x^n d\mu(x)\Big|_p \leq C (1/p)^{\alpha(e)(r-n)} \text{ } \text{ }\text{ for } \infty \in U(e), 0 \notin U(e), 0 \leq n \leq 2k-2,
\end{equation}
where  $\alpha (e)= -\text{inf}_{u,v \notin U(e)}\{ val_p(u-v) \}$ for $\infty \in U(e)$.
Then the distribution can be extended uniquely to the space $A_{2k}$ of $\C_p$-valued functions on $\PQp$ which are locally analytic except for a pole at $\infty$ of order at most $2k-2$. The space $A_{2k}$ is endowed with a weight $2k-2$ action of $\Gamma$ defined by the same formula (\ref{action on poly}) giving the weight $2k-2$ action on $\mathcal{P}_{2k-2}$.

We will show in the following section that if $c$ is a harmonic cocycle satisfying the conditions (\ref{conditions}), then the distribution given by $c$ satisfies the two bounds above.

\vspace{4mm}

\subsection{The computation of the bounds}
\label{The case of balls centered at zero}
\vspace{6mm}
Recall that $\Gamma$ acts on $ C_{\text{har}}(\mathcal{P}_{2k-2}^{\vee})$ by $(c|\gamma)(e)=c(\gamma e)|\gamma$. Hence we have
\begin{equation}
\label{equivariance}
\int_{U(e)}P(t)d\mu_{c|\gamma}(t)=(c|\gamma)(e)(P)=c(\gamma e)(P|\gamma^{-1})=\int_{\gamma U(e)}(P|\gamma^{-1})(t)d\mu_{c}(t).
\end{equation}
We will use this property to prove the lemmas below.

\begin{lemma}
Let $c$ be an element of $ C_{\text{har}}(\mathcal{P}_{2k-2}^{\vee})$ satisfying condition $c| D=c$ from (\ref{conditions}). Then inequality (\ref{no inf}) holds for $\mu_c$ and $a=0$.
\end{lemma}
\begin{proof}
All balls of $\PQp$ centered at zero can be written as translates of $\Z_p$ or $p\Z_p$ via some power of the matrix $D$. Let $e$ such that $U(e)=\Z_p$ and consider for example $U(D^me)=\{\tau| \:val_p(\tau) \geq 2m   \}$ for some $m \in \Z$. Then property (\ref{equivariance}) of the integral combined with the $D^m$-invariance of $c$ gives
$$
\int_{U(D^me)}x^n|D^{-m}d \mu_c=\int_{U(e)}x^n d \mu_c,
$$
which is
$$
\int_{U(D^me)}x^nd \mu_c=(1/p)^{-2m(n-(2k-2)/2)}\int_{U(e)}x^n d \mu_c.
$$

If instead than $U(e)=\Z_p$ we take $U(e)=p\Z_p$, we get similar inequalities for the balls that have radius given by an odd valuation. Condition  (\ref{no inf}) then follows for all balls centered at zero and $0 \leq n \leq 2k-2$ .
\end{proof}

\begin{lemma}
Let $c$ be an element of $ C_{\text{har}}(\mathcal{P}_{2k-2}^{\vee})$ satisfying condition $c|(1+S)=0$ from (\ref{conditions}). Assume also that inequality (\ref{no inf}) holds for $\mu_c$ with $a=0$. Then inequality (\ref{yes inf}) is also satisfied.
\end{lemma}
\begin{proof}
Any ball of $\PQp$ centered at $\infty$ can be written as $U(Se)$ for some ball $U(e) \subseteq \Q_p$ centered at zero. Then property (\ref{equivariance}) of the integral together with the fact $c=-c|S$ gives
$$
\int_{U(e)}x^nd \mu_c=-\int_{U(Se)}(x^n|S) d\mu_c.
$$
So we get
$$
\int_{U(Se)}x^n d\mu_c=(-1)^{n+1}\int_{U(e)}x^{2k-2-n}d \mu_c,
$$
hence
$$
\Big |  \int_{U(Se)}x^n d\mu_c\Big |_p \leq C\Big(\frac{1}{p}\Big)^{\alpha(e)(2k-2-n-(2k-2)/2)}=C\Big(\frac{1}{p}\Big)^{\alpha(e)((2k-2)/2-n)}=C\Big(\frac{1}{p}\Big)^{\alpha(Se)((2k-2)/2-n)},
$$
and the thesis follows (recall that $\alpha(e)$ and $\alpha(Se)$ have been defined at the end of the previous section).
\end{proof}

\begin{lemma}
Let $c$ be an element of $ C_{\text{har}}(\mathcal{P}_{2k-2}^{\vee})$ satisfying conditions $c|(1+S)=0$ and  $c| D=c$ from (\ref{conditions}). Then inequality (\ref{no inf}) holds for any $a \in \mathbb{Q}_p$.
\end{lemma}
\begin{proof}
Let $\alpha=\begin{pmatrix}p^m &a/p^m\\ 0 & p^{-m}  \end{pmatrix}$ for an integer $m$. A general ball in $\Q_p$ can be written as $\alpha( \Z_p)$ or $\alpha( p\Z_p)$.

Let $U(e)=\Z_p$ so that $U(\alpha e)=a+p^{2m}\Z_p$. From  property (\ref{equivariance}) of the integral we get
$$
\int_{U(\alpha e)}(x-a)^n d \mu_c= \int_{U(e)}((x-a)^n |\alpha )d \mu_{c|\alpha}=p^{2m(n-(2k-2)/2)}\int_{U(e)}x^n d\mu_{c|\alpha}.
$$
Recall now that we are implicitly identifying $\text{MS}^{\Gamma}( C_{\text{har}}(\mathcal{P}_{2k-2}^{\vee}))$  with the subset of $ C_{\text{har}}(\mathcal{P}_{2k-2}^{\vee})$ given by the  harmonic cocycles $c$ satisfying the relations (\ref{conditions}). So, with a slight abuse of notation, $d\mu_c=d\mu_{c\{0,\infty\}}$ and $d \mu_{c|\alpha}=d \mu_{c\{\alpha^{-1}0, \alpha^{-1}\infty\}}$. But the pair $(\alpha^{-1}0, \alpha^{-1}\infty)$ can be written as a finite sum of unimodular pairs, so we get a finite sum
$$
\int_{U(\alpha e)}(x-a)^n d \mu_c=p^{2m(n-(2k-2)/2)}\sum_t \int_{U(e)}x^n d\mu_{c\{r_t, s_t  \}},
$$
where the pairs $(r_t,s_t)$ are unimodular. Now recall that SL$_2(\Z)$ acts transitively on unimodular pairs, therefore the sum above can be rewritten as
$$
 p^{2m(n-(2k-2)/2)}\sum_t\int_{U(e)}x^n d\mu_{c\{0, \infty \}|\beta_t}= p^{2m(n-(2k-2)/2)}\sum_t\int_{U(e)}x^n d\mu_{c|\beta_t},
$$
for some $\beta_t \in $ SL$_2(\Z)$.
Using again property (\ref{equivariance}) of the integral we get
\begin{equation}
\label{sum of integrals}
 p^{2m(n-(2k-2)/2)}\sum_t\int_{U(e)}x^n d\mu_{c|\beta_t}=p^{2m(n-(2k-2)/2)}\sum_t \int_{U(\beta_te)}P_t(x)d\mu_c,
\end{equation}
where $P_t(x)=x^n|\beta_t^{-1}$ is a polynomial in $\mathcal{P}_{2k-2}$. Now note that $U(\beta_te)=U(e)$ as SL$_2(\Z)$ stabilizes $\Z_p$. This means that inequality (\ref{no inf}) holds for the integrals in the sum above. The norm of these integrals will be bounded by some constant $Cp^0=C$ because the integrals are on $\Z_p$, and so for $U(e)=\Z_p$ we get
$$
\Big | \int_{U(\alpha e)}(x-a)^n d \mu_c  \Big |_p\leq C(1/p)^{2m(n-(2k-2)/2)}.
$$
If instead we have $U(e)=p\Z_p$, the proof is the same as above, we just need to be more careful as $\beta_t$ does not necessarily fix $p\Z_p$. However, as  $\beta_t\in$ SL$_2(\Z)$, we have that  $\beta_t(p\Z_p)$ will also be a ball of radius $1/p$, so the integral in (\ref{sum of integrals}) is bounded by $(1/p)^{-(2k-2)/2}$ and the proof proceedes as in the previous case.
\end{proof}

\vspace{4mm}

The lemmas above imply that the integral defined by $c\{r,s\}$ with $c \in \text{MS}^{\Gamma}( C_{\text{har}}(\mathcal{P}_{2k-2}^{\vee}))$ can be extended  uniquely to the space $A_{2k}$ of $\C_p$-valued functions on $\PQp$ which are locally analytic except for a pole at $\infty$ of order at most $2k-2$. So the map
$$
  c\{r,s\}      \mapsto f\{r,s\}(z)=\int_{\PQp} \frac{1}{z-t} d \mu_{c\{r,s\}}(t)
$$ 
makes sense. 

\vspace{4mm}

\subsection{The end of the proof}
\label{The end of the proof}
\vspace{6mm}
In this short section we prove that the function $f\{r,s\}(z)$ defined in (\ref{def Pois Ker}) is rigid analytic and that $f\{r,s\}$ is a $\Gamma$-invariant modular symbol. Most of the material is taken from \cite{Tei}.
\begin{proposition}
Let $c$ be an element of $\MS ^{\Gamma}( C_{\text{har}}(\mathcal{P}_{2k-2}^{\vee})) $. The function
$$
f\{r,s\}(z)=\int_{\PQp} \frac{1}{z-t} d \mu_{c\{r,s\}}(t)
$$
is an element of $\mathcal{A}_{2k}$.
\end{proposition}
\begin{proof}
Note that the index $\{r,s\}$ can be dropped for this proof. Then the statement is shown in the proof of Theorem 3 of \cite{Tei}.
\end{proof}

\begin{lemma}
\label{miracle computation}
The quantity
$$
\frac{(ct+d)^{2k-2}}{\gamma z- \gamma t}-\frac{(cz+d)^{2k}}{z-t}
$$
is a polynomial in $t$ of degree at most $2k-2$.
\end{lemma}
\begin{proof}
See the proof of Theorem 3 in \cite{Tei}.
\end{proof}

\begin{proposition}
The expression $f\{r,s\}$ defined in (\ref{def Pois Ker}) gives an element of $\MS^{\Gamma}(\mathcal{A}_{2k})$.
\end{proposition}
\begin{proof}
We have that $f\{r,s\}+f\{s,t\}=f\{r,t\}$ because to define $f\{r,s\}$ we used a cocycle $c\{r,s\}$ which satisfies the modular symbol condition. Now we need to show that $f\{\gamma r,\gamma s\}(z)=(f\{r,s\}(z))|\gamma^{-1}$. By property (\ref{equivariance}) of the integral we get
$$
\int_{\PQp}\frac{1}{z-t}d\mu_{c\{\gamma r,\gamma s\}}(t)= \int_{\PQp}\frac{(ct+d)^{2k-2}}{z-\gamma t} d \mu_{c\{r,s\}}(t).
$$
Using now Lemma \ref{miracle computation} we get that this expression is
$$
(c\gamma^{-1}z+d)^{2k}\int_{\PQp}\frac{1}{\gamma^{-1}z-t} d \mu_{c\{r,s\}}(t)=(a-cz)^{-2k}\int_{\PQp}\frac{1}{\gamma^{-1}z-t} d \mu_{c\{r,s\}}(t),
$$
which is $(f\{r,s\}(z))|\gamma^{-1}$, so the thesis follows.
\end{proof}

\vspace{4mm}

\section{The residue of $J_{k,D}$}
\label{res map}
Our goal for this section is to define a left inverse for $\ST$ and evaluate it at $J_{k,D}$, assuming $\big (\frac{D}{p}\big)=1$. The definition of this map is inspired by \cite{Tei} and \cite{Sch}, and we extend their construction to the setting of $\Gamma$-invariant modular symbols. The computation involving $J_{k,D}$ is carried out in Section \ref{computation residue}. Consider the pairing on $\mathcal{P}_{2k-2}$ given by
\begin{equation}
\label{pairing}
\langle T^i, T^j   \rangle={2k-2 \choose i}^{-1}(-1)^i \delta_{i,2k-2-j}.
\end{equation}

\begin{definition}
Let $c \in C_{har}(\mathcal{P}_{2k-2})$. For any edge $e$ of $\T$ we define the components $c_i(e)$ of $c$ by the following expression
$$
c(e)(T)=\sum_{i=0}^{2k-2} {2k-2 \choose i}c_i(e)T^i.
$$
\end{definition}

If we identify $\mathcal{P}_{2k-2}$ and $\mathcal{P}_{2k-2}^{\vee}$ via the pairing above, then the distribution defined by
$$
\int_{U(e)} x^i d\mu_c(t)=(-1)^ic_{2k-2-i}(e)
$$
for a cocycle $c \in C_{har}(\mathcal{P}_{2k-2})$ agrees with the distribution that was given in (\ref{distribution dual}) for a cocycle in $C_{har}(\mathcal{P}_{2k-2}^{\vee})$. Because of what we did in Section \ref{ST}, this distribution can be extended to $A_{2k}$ and we can define a map
$$
\ST:  \MS^{\Gamma} (C_{har}(\mathcal{P}_{2k-2}))\rightarrow \MS^\Gamma(\mathcal{A}_{2k})
$$
in the same way as we defined $\ST^{\vee}$ in Section \ref{ST}. So we have the following commutative diagram.

\[ \begin{tikzcd}
 \MS^{\Gamma} (C_{har}(\mathcal{P}_{2k-2})) \ar[r, "\sim"] \arrow[swap]{d}{\ST} &  \MS^{\Gamma} (C_{har}(\mathcal{P}_{2k-2}^{\vee})) \arrow{d}{\ST^{\vee}} \\%
\MS^\Gamma(\mathcal{A}_{2k}) \rar[ equal] & \MS^\Gamma(\mathcal{A}_{2k})
\end{tikzcd}
\]

\vspace{2mm}

We now define a map $\Res$ on $ \MS^\Gamma(\mathcal{A}_{2k})$ with values in $ \MS(C_{har}(\mathcal{P}_{2k-2}))$ by $f\{r,s\} \mapsto c\{r,s\}$ with
\begin{equation}
\label{def res map}
c\{r,s\}(e)(T)=\sum_{i=0}^{2k-2}{2k-2 \choose i}(-1)^{i}\Res_e (z^{2k-2-i}f\{r,s\}(z) dz)T^i,
\end{equation}
where $\Res_e$ denotes the residue with respect to the annulus in $\Hp$ corresponding to the edge  $e$.
The lemma below shows that $\Res(f\{r,s\})$ is a $\Gamma$-invariant modular symbol, hence we constructed a map
$$
\Res : \MS^\Gamma(\mathcal{A}_{2k}) \rightarrow \MS^{\Gamma}(C_{har}(\mathcal{P}_{2k-2})).
$$

\begin{lemma}
The modular symbol defined in (\ref{def res map}) is $\Gamma$-invariant.
\end{lemma}
\begin{proof}
For any $\gamma=\begin{pmatrix} a &b \\ c& d \end{pmatrix}$ in $\Gamma$ we need to show that the expression
\begin{equation}
\label{sx}
c\{\gamma r,\gamma s\}=\sum_{i=0}^{2k-2}{2k-2 \choose i}(-1)^i\Res_e(z^{2k-2-i}f\{\gamma r,\gamma s\}(z) dz)T^i
\end{equation}
equals the expression 
\begin{equation}
\label{dx}
(c\{r, s\})|\gamma^{-1}=\sum_{i=0}^{2k-2}{2k-2 \choose i}(-1)^i\Res_{\gamma^{-1}e}(z^{2k-2-i}f\{ r,s\}(z) dz)(T^i|\gamma^{-1}).
\end{equation}
We can rewrite (\ref{sx}) as
\begin{equation}
\label{sx bis}
\sum_{i=0}^{2k-2}{2k-2 \choose i}(-1)^i\Res_{\gamma^{-1}e}\Big((\gamma z)^{2k-2-i}(-c\gamma z+a)^{-k}f\{ r,s\}(z) \frac{dz}{(cz+d)^2}\Big)T^i
\end{equation}
because of the property of the annular residue $\Res_{\gamma^{-1} e}(f(\gamma z)d(\gamma z))=\Res_e(f(z)dz)$.
The lemma follows after completely expanding (\ref{dx}), (\ref{sx bis}) and comparing their coefficients.
\end{proof}

\begin{proposition}
The map $\Res$ is a left inverse for $\ST$, i.e. $\Res \circ \ST=\Id$.
\end{proposition}
\begin{proof}
Note that the index $\{r,s\}$ can be dropped, as it is not needed here. The statement then follows from the proof of Theorem 3 of \cite{Tei}.
\end{proof}

Recall that there is an isomorphism between $ \MS^{\Gamma}(C_{har}(\mathcal{P}_{2k-2}))$ and $\MS ^{\Gamma_0(p)}(\mathcal{P}_{2k-2})$ which is induced by evaluating a harmonic cocycle at the standard edge $e_0$. This gives a map
$$
\Res_0: \MS^\Gamma(\mathcal{A}_{2k}) \rightarrow \MS ^{\Gamma_0(p)}(\mathcal{P}_{2k-2})
$$ 
which is the composite of $\Res$ and this isomorphism. In the next section we will compute $\Res_0(J_{k,D})$. 

\vspace{4mm}

\subsection{The computation of the residue}
\label{computation residue}
\vspace{6mm}
Let $\mathcal{F}_D(\Z)$ denote, as before, the set of integral binary quadratic forms of discriminant $D$. A form $Q(x,y)=ax^2+bxy+cy^2 \in \mathcal{F}_D(\Z)$ is called a \emph{Heegner form} if $p$ divides $a$, and the set of all Heegner forms of discriminant $D$ is denoted $\mathcal{F}_D^{(p)}(\Z)$. The \emph{standard annulus} of $\Hp$ is the annulus of $\Hp$ which reduces to the standard vertex $e_0$ of $\T$. Recall that in this section we are assuming $\big (\frac{D}{p} \big )=1$. We will need the proposition below. 
\begin{proposition}
Let $Q\in \mathcal{F}_{Dp^{2n}}(\Z)$ be a binary quadratic form with integer coefficients and discriminant $Dp^{2n}$. Then $\Res_{e_0}Q(z,1)^{-k}$ is not zero if and only if $Q$ is an Heegner form and $p$ does not divide the discriminant of $Q$.
\end{proposition}
\begin{proof}
This follows from the lemmas below and the $p$-adic Residue Theorem (see for example \cite{FVdP}).
\end{proof}

\begin{lemma}
Let $Q=[A, B, C]$ and let $r_1$ and $r_2$ be the two roots of $Q$. Assume that $val_p(r_1) \geq 0$ (resp. $val_p(r_2) \geq 0$) and $val_p(r_2) \leq -1$ (resp. $val_p(r_1) \leq -1$), i.e. one root is \enquote{inside} the standard annulus in $\Hp$ and the other one is \enquote{outside}. Then $p|A$.
\end{lemma}
\begin{proof}
The sum of the two roots is
$$
r_1+r_2=\frac{-B+\sqrt{\Delta}}{2A}+\frac{-B-\sqrt{\Delta}}{2A}=\frac{-B}{A},
$$
where $\Delta$ is the discriminant of $Q$.
Therefore $val_p(-B/A)=val_p(r_1+r_2)<0$ and $p|A$.
\end{proof}

\begin{lemma}
Let $Q\in \mathcal{F}_{Dp^{2n}}(\Z)$ with $Q=[A, B, C]$ and let $\Delta$ be the discriminant of $Q$. If $Q$ is an Heegner form and if $p|\Delta,$ then $\Res_{e_0}(Q(z,1)^{-k})=0$.
\end{lemma}
\begin{proof}
We know that $p|A$, so we can write $A=ap^\alpha$ where $p \nmid a$ and $\alpha>0$. As $p$ divides the discriminant then we also have $p|B$ and $B=bp^\beta$ with $p\nmid b$ and $\beta>0$. There are now three cases.

If $2\beta>\alpha$ then we can write $\Delta=p^{\alpha}(b^2p^{2\beta-\alpha}-4aC)$. Note that $\alpha$ must be even. The roots of $Q$ can be written as
$$
\frac{-bp^\beta \pm p^{\alpha/2}\sqrt{b^2p^{2\beta-\alpha}-4aC}}{2ap^\alpha},
$$
so the $p$-adic valuation of both roots is $-\alpha/2\leq -1$. Hence both $r_1$ and $r_2$ are \enquote{outside} the standard annulus and $\Res_{e_0}(Q(z,1)^{-k})=0$ by the Theorem of Residues in $\Hp$.

If $2\beta<\alpha$ then we have $\Delta=p^{2\beta}(b^2-4ap^{\alpha-2\beta}C)$. We can aslo write $\alpha=2\beta +x$ fore some $x>0$ and the two roots of $Q(z,1)$ are 
$$
\frac{-bp^\beta \pm p^\beta \sqrt{b^2-4ap^xC}} {2ap^{2\beta +x}}=\frac{-b \pm \sqrt{b^2-4ap^xC}} {2ap^{\beta +x}}.
$$
Note that $p\nmid C$, so the valuation of one of the roots must be $ -\beta <0$. This implies that the valuation of the other root must be $-\beta-x$, because $val_p(r_1r_2)=-x-2\beta$. Therefore also in these case the valuation of both roots is at most $-1$, hence they are \enquote{outside} the standard annulus and the residue of $Q(z,1)^{-k}$ is zero also in this case.

Finally let us consider the case $2\beta=\alpha$. We have $\Delta=p^{2\beta}(b^2-4aC)$ and the two roots can be written as
$$
\frac{-b\pm \sqrt{b^2-4aC}}{2ap^\beta}.
$$
Then $val_p(-b+\sqrt{b^2-4aC})=0$, which implies that $val_p(-b-\sqrt{b^2-4aC})=0$. Therefore the valuation of both roots is $-\beta$ and the residue of $Q(z,1)^{-k}$ is again zero.
\end{proof}

\begin{lemma}
Let $Q$ be a Heegner form and assume $p\nmid \Delta$. Then one of the roots of $Q$ is \enquote{inside} the standard annulus of $\Hp$ and the other one is \enquote{outside}.
\end{lemma}
\begin{proof}
We let again $Q=[A,B,C]$, with $A=p^\alpha a$ and $p\nmid a$. Then either $val_p(-B+\sqrt{\Delta})=\alpha$ or $val_p(-B-\sqrt{\Delta})=\alpha$.  If we consider the product $r_1r_2$ we see that in the first case $val_p(r_1)=0$ and $val_p(r_2)=-\alpha$, while in the second case $val_p(r_1)=-\alpha$ and $val_p(r_2)=0$ .
\end{proof}

\vspace{4mm}

In order to make a distinction between the cases in which a root is inside or outside the standard annulus, it will be important to fix a choice of a square root of $D$ in $\Q_p$ and to introduce the notation of first (resp. second) root of $Q$ \emph{in a $p$-adic sense}. Note that we are still assuming that $D$ is a square modulo $p$.

\begin{definition}
Let $\sqrt{D}$ be a fixed choice of a square root of $D$ in $\Q_p$ and let $Q=ax^2+bxy+cy^2$ be a binary quadratic form with coefficients in $\Q_p$. The first root $r_1$ and the second root $r_2$ of $Q$ in a $p$-adic sense are defined as
$$
r_1=\frac{-b+\sqrt{D}}{2a}, \:\:\:\: r_2=\frac{-b-\sqrt{D}}{2a} .
$$
\end{definition}

\begin{remark}
The notion in the definition above is different from the notion of first and second root used so far and given in the introduction, because it depends on a square root of $D$ in $\Q_p$ rather  than on the positive real root. This notion will be only used in Definition \ref{padic intersection} below.
\end{remark}

\begin{definition}
\label{padic intersection}
Let $Q$ be as in the definition above and assume $p|A$. Let $r_1,r_2$ be the fisrt and second root of $Q$ in a $p$-adic sense. We will denote by $(\gamma_Q \cdot e_0)$ the \enquote{$p$-adic intersection number} defined as
$$
(\gamma_Q \cdot e_0)= \left\{
\begin{array}{ll}
      1 & \text{if } val_p(r_1)\geq 0, \:\: val_p(r_2) \leq -1,  \\
      -1 & \text{if } val_p(r_1)\leq -1, \:\: val_p(r_2) \geq 0.\\
\end{array} 
\right.
$$
In other words $(\gamma_Q \cdot e_0)=1$ (resp. $-1$) if $r_1$ is inside the standard annulus (resp. outside) and $r_2$ outside (resp. inside). Note that if we change the choice of the square root of $D$ in $\Q_p$ then $(\gamma_Q \cdot e_0)$ changes sign.
\end{definition}

\vspace{4mm}

\begin{proposition}
\label{proof is MS}
Let $k\ge 1$ be odd. For all $\{r,s\}\in \mathbb{P}^1({\Q})\times \mathbb{P}^1({\Q})$, the finite sum
$$
\kappa_{k,D}\{r,s\}(z)={2k-2 \choose k-1}\frac{1}{D^{k-1}\sqrt{D}}\sum_{Q  \in \mathcal{F}_D^{(p)}(\Z)}(\gamma_Q\cdot (r,s))(\gamma_Q \cdot e_0)\cdot Q(z,1)^{k-1}
$$ 
is a polynomial of degree $2k-2$ with coefficients in $\C_p$. The function
$$
\kappa_{k,D}: \mathbb{P}^1({\Q})\times \mathbb{P}^1({\Q})\rightarrow \mathcal{P}_{2k-2}
$$
is an element of $\MS^{\Gamma_0(p)}(\mathcal{P}_{2k-2})$.
\end{proposition}
\begin{proof}
This is true because Heegner forms are fixed by $\Gamma_0(p)$ and because $(\gamma_Q\cdot e_0)=(\gamma_{Q_{\delta}}\cdot e_0)$, where  $Q_{\delta}=Q|\delta \in \mathcal{F}_D^{(p)}$ and the action of $\Gamma_0(p)$ on binary quadratic forms is the usual one. Indeed let $r_1$ and $r_2$ be the first and second root of $Q$. Then one can check that $\delta^{-1}r_1$ and $\delta^{-1}r_2$ are the first and second root of $Q_{\delta}$, respectively. Because the matrix $\delta^{-1}$ fixes the standard affinoid of $\Hp$, the $p$-adic intersection number does not change.
\end{proof}

\vspace{4mm}

We are now ready to compute the residue of $J_{k,D}$.

\vspace{4mm}

\begin{theorem}
\label{theorem residue}
Let $k\ge 1$ be odd and let $\:\Res_0: \MS^\Gamma(\mathcal{A}_{2k}) \rightarrow \MS ^{\Gamma_0(p)}(\mathcal{P}_{2k-2})$ be the map defined in Section \ref{res map}.  Then
$$
\Res_{0}(J_{k,D})=\kappa_{k,D}.
$$
\end{theorem}
\begin{proof}
Let $P:=\Res_0(J_{k,D})$, so
$$
P\{r,s\}(T)=\sum_{i=0}^{2k-2}{2k-2 \choose i}(-1)^i\Res_{e_0}(z^{2k-2-i}J_{k,D}\{r,s\}(z)dz)T^i.
$$
We want to show that the expression above equals $\kappa_{k,D}\{r,s\}(T)$. We will drop the index $\{r,s\}$ as it is not necessary in this proof. Given a Heegner form $Q=[A,B,C]$, let $H(T)=Q(T,1)^{k-1}$ and let $H^{(2k-2-i)}$ denote the coefficient of degree $2k-2-i$ of the polynomial $H(T)$. Because only the Heegner forms give a contribution in the computation of the residue, it is enough to show that
$$
{2k-2 \choose i}(-1)^i\Res_{e_0}\Big(\frac{z^i}{Q(z,1)^k}\Big)=(\gamma_q\cdot e_0)\frac{H^{(2k-2-i)}}{D^{k-1}\sqrt{D}}{2k-2 \choose k-1}.
$$
From now on we will assume that $(\gamma_q\cdot e_0)=1$, so that $\Res_{e_0}(z^i/Q(z,1)^k)=\Res_{r_1}(z^i/Q(z,1)^k)$. This is enough because if it was $(\gamma_q\cdot e_0)=-1$, then we would just take the residue with respect to $r_2$, getting the opposite sign. We can write
$$
\Res_{e_0}\Big (\frac{z^i}{Q(z,1)^k}\Big)=\sum_{l=0}^{M}{i \choose l}\frac{r_1^{i-l}}{A^k}\Res_{e_0}\Big(\frac{(z-r_1)^{l-k}}{(z-r_2)^k}\Big),
$$
where $M=\text{min}(i,k-1)$. The upper bound for $l$ is $M$ because $\Res_{e_0}((z-r_1)^{l-k}/(z-r_2)^k)=0$ if $l \geq k$.

To compute the residues in the sum above we express $(z-r_1)^{l-k}/(z-r_2)^k$ as a power series in $(z-r_1)$, getting
$$
\frac{(z-r_1)^{l-k}}{(z-r_2)^k}=\frac{(z-r_1)^{l-k}}{(r_1-r_2)(k-1)!}\sum_{j=k-1}^{\infty}\frac{j(j-1)(j-2) ... (j-k+2)(z-r_1)^{j-k+1}}{(r_2-r_1)^j}.
$$
This expression can be obtained by formally differentiating the geometric series which gives the expansion of $(z-r_2)^{-1}$. Therefore
$$
\Res_{e_0}\Big(\frac{(z-r_1)^{l-k}}{(z-r_2)^k}\Big )={2k-2-l \choose k-1}\frac{(-1)^l}{(r_1-r_2)^{2k-1-l}},
$$
and
$$
\Res_{e_0}\Big (\frac{z^i}{Q(z,1)^k}\Big)=\sum_{l=0}^{M}{i \choose l}{2k-2-l \choose k-1}\frac{(-1)^{l}r_1^{i-l}}{A^k}\Big(\frac{A}{\sqrt{D}}\Big )^{2k-1-l},
$$
where we used the equality $r_1-r_2=\sqrt{D}/A$.

On the other hand, we have
\begin{equation}
\label{coefficient}
H^{(2k-2-i)}=A^{k-1}\sum_{l=0}^{N} {k-1 \choose l}{k-1 \choose 2k-2-i-l}(-r_1)^{k-1-l}(-r_2)^{i+l-k+1}    ,
\end{equation}
where $N=\text{min}(2k-2-i,k-1)$. Now using the equality $r_2=r_1-\sqrt {D}/A$ we can rewrite (\ref{coefficient}) as
$$
A^{k-1}\sum_{l=0}^{N} \Big ({k-1 \choose l}{k-1 \choose 2k-2-i-l}r_1^{k-1-l}(-1)^i \sum_{h=0}^{i+l-k+1}{i+l-k+1 \choose h}r_1^h \Big(\frac{\sqrt{D}}{-A}\Big)^{i+l-k+1-h}\Big ).
$$
From the above computations we see that to complete the proof we need to show that
\begin{equation}
\label{PolyTeit}
{2k-2 \choose i}\sum_{l=0}^{M}{i \choose l}{2k-2-l \choose k-1}(-1)^lr_1^{i-l} \Big(\frac{\sqrt{D}}{A}     \Big)^l
\end{equation}
is equal to
\begin{equation}
\label{PolyDar}
{2k-2 \choose k-1}      \sum_{l=0}^{N} \Big ({k-1 \choose l}{k-1 \choose 2k-2-i-l}r_1^{k-1-l}  \sum_{h=0}^{i+l-k+1}{i+l-k+1 \choose h}r_1^{i+l-k+1-h}(-1)^h \Big(\frac{\sqrt{D}}{A}\Big)^{h}   \Big )  .
\end{equation}

\vspace{2mm}

We can see the two expressions above as polynomials in $(\sqrt{D}/A)$. These polynomials have the same degree, because if $M=k-1$ then $N=2k-2-i$, while if $M=i$ then $N=k-1$. Assume $M=k-1$ (the case $M=i$ can be treated similarly).

The coefficient of degree $l$ of the polynomial in (\ref{PolyTeit}) is
\begin{equation}
\label{CoeffTeit}
{2k-2 \choose i}{i \choose l}{2k-2-l \choose k-1}(-1)^lr_1^{i-l} ,
\end{equation}
while the coefficient of degree $l$ of the polynomial in (\ref{PolyDar}) is
\begin{equation}
\label{CoeffDar}
{2k-2 \choose k-1} \sum_{t=0}^{2k-2-i}{k-1 \choose t} {k-1 \choose 2k-2-i-t} {i+t-k+1 \choose l}(-1)^l r_1^{i-l}.
\end{equation}
Therefore to conclude the proof it is enough to show the equality
\vspace{2mm}
\begin{equation}
\label{Equality}
{2k-2 \choose k-1}\sum_{t=0}^{2k-2-i}{k-1 \choose t} {k-1 \choose 2k-2-i-t} {i+t-k+1 \choose l}={2k-2 \choose i}{i \choose l}{2k-2-l \choose k-1}.
\end{equation}
\vspace{2mm}

We will prove this equality by applying some binomial identities to its left hand side.
If we let $a:=k-1$ and $b:=2k-2-t-i$, then the left hand side of (\ref{Equality}) becomes
\vspace{2mm}
\begin{equation}
\label{MidStep}
{2k-2 \choose k-1}\sum_{t=0}^{2k-2-i}{a-b \choose l}{a \choose b}{a \choose t}={2k-2 \choose k-1}{k-1 \choose l}\sum_{t=0}^{2k-2-i}{k-1-l \choose 2k-2-i-t}{k-1 \choose t},
\end{equation}
where the equality follows from the binomial identity ${a \choose b}{a-b \choose l}={a \choose l}{a-l \choose b}$. 

Now we let $n=2k-2-i$ and (\ref{MidStep}) becomes 
\vspace{2mm}
\begin{equation}
\label{LastStep}
{2k-2 \choose k-1}{k-1 \choose l}\sum_{t=0}^{n}{k-1-l \choose n-t}{k-1 \choose t}={2k-2 \choose k-1}{k-1 \choose l}{2k-2-l \choose 2k-2-1},
\end{equation}
where the last equality holds because of the Chu-Vandermonde indentity, i.e.
$$
\sum_{t=0}^{n}{m \choose t}{s-m \choose n-t}={s \choose n},
$$
with $m=k-1-l$ and $s=2k-2-l$ in our case.

Now it is easy to see that (\ref{LastStep}) is equal to the right hand side of (\ref{Equality}), so the theorem follows.
\end{proof}

\vspace{4mm}

\section{A Zagier form of level $p$}
\label{period computation}
Let $k$ continue to be an \emph{odd} integer, and assume also $k\geq 3$. The Zagier form $f_{k,D}(z)$ of the introduction admits an analogue in level $p$, given by
$$
f_{k,D}^{(p)}(z)=\sum_{Q\in \mathcal{F}_D^{(p)}(\Z)}\frac{(\gamma_Q\cdot e_0)}{Q(z,1)^k},
$$
where $ \mathcal{F}_D^{(p)}(\Z)$ was defined at the beginning of Section \ref{computation residue}. The coefficient $(\gamma_Q\cdot e_0)$ ensures that there is no cancellation between the forms $Q$ and $-Q$. Moreover, this coefficient will play an important role in the proof of Theorem \ref{complex generating series}.

\begin{proposition}
The function $f_{k,D}^{(p)}(z)$ is a weight $2k$ cusp form for $\Gamma_0(p)$.
\end{proposition}
\begin{proof}
This follows from the fact that $\Gamma_0(p)$ fixes Heegner forms. Indeed for any $\delta=\begin{pmatrix} a & b \\ c& d \end{pmatrix}$ in $\Gamma_0(p)$ we have
$$
(f_{k,D}^{(p)}|\delta)(z)=\sum_{Q\in \mathcal{F}_D^{(p)}(\Z)}\frac{(\gamma_Q\cdot e_0)}{Q_{\delta}(z,1)^k},
$$
where $Q_{\delta}=Q|\delta \in \mathcal{F}_D^{(p)}$ and the action of $\Gamma_0(p)$ on binary quadratic forms is the usual one. Moreover when we proved Proposition \ref{proof is MS} we showed that $(\gamma_Q\cdot e_0)=(\gamma_{Q_{\delta}}\cdot e_0)$. This completes the proof.
\end{proof}

An \emph{Eichler cocycle} of weight $2k$ is an element of the space $\MS^{\SL_2(\mathbb{Z})}(\mathcal{P}_{2k-2})$ of $\SL_2(\Z)$-invariant modular symbols with values in $\mathcal{P}_{2k-2}$. More generally, an Eichler cocycle of weight $2k$ and level $p$ is an element of the space $\MS^{\Gamma_0(p)}(\mathcal{P}_{2k-2})$. The relevance of Eichler cocycles to modular forms arises from the Eichler-Shimura isomorphism, which to any cusp form $f$ of weight $k$ for a congruence group $\Gamma$ associates the Eichler cocycle of weight $k$ defined by
\begin{equation}
\label{period def}
\kappa_f\{r,s\}:=\int_r^sf(z)(x-z)^{k-2}dz,
\end{equation}
where the integral is over the geodesic in the upper half plane joining $r$ and $s$. The right hand side is a polynomial in $x$, and $\kappa_f$ is an element of $\MS^{\Gamma}(\mathcal{P}_{2k-2})$. Furthermore, the assignment $f \mapsto \kappa_f$ induces a Hecke equivariant vector space isomorphism between the space $S_k(\Gamma)$ of cusp forms of weight $k$ for $\Gamma$ and the space $\MS^{\Gamma}(\mathcal{P}_{2k-2})$. Some references for this material are \cite{greenberg-stevens} (Section 4), \cite{KZ2}, and \cite{rational} (Chapter 2).

Let $f$ be a weight $2k$ cusp form for $\Gamma_0(p)$ (or $\SL_2(\Z)$). To $f$ we associate the modular symbol
$$
\bar{\kappa}_f\{r,s\}:= {\kappa}_f\{r,s\}- {\kappa}_f\{-r,-s\}|\widetilde{I},
$$
where
$$
\widetilde{I}:=\begin{pmatrix} -1 &0 \\0&1\end{pmatrix}.
$$

\begin{lemma}
The modular symbol $\bar{\kappa}_f$ is $\Gamma_0(p)$-invariant, i.e. it belongs to $\MS^{\Gamma_0(p)}(\mathcal{P}_{2k-2})$.
\end{lemma}
\begin{proof}
Let $\gamma\in\Gamma_0(p)$. The result follows from the following equalities
\begin{align*}
{\kappa}_f\{-\gamma r,-\gamma s\} &={\kappa}_f\{\widetilde{I}\gamma r,\widetilde{I}\gamma s\}\\
&={\kappa}_f\{\widetilde{I}\gamma \widetilde{I}(- r),\widetilde{I}\gamma\widetilde{I}(- s)\}\\
&={\kappa}_f\{- r,-s\}|\widetilde{I}\gamma^{-1}\widetilde{I},
\end{align*}
where we used the fact that $\widetilde{I}\gamma\widetilde{I}\in\Gamma_0(p)$ in the last equality.
\end{proof}

An important theme of \cite{KZ2} is that the forms $f_{k,D}(z)$ have \emph{rational periods}. Indeed Kohnen and Zagier computed the \emph{even} period of $f_{k,D}(z)$, which is essentially equal to $\kappa_f\{0,\infty\}$, up to a multiplicative constant. They did not use the notation of modular symbols, however their result can be formulated as
\begin{equation}
\label{koh zag}
\kappa_{f_{k,D}}\{0,\infty\}={2k-2 \choose k-1 }\frac{\pi}{D^{k-1}\sqrt{D}}\sum_{Q\in\mathcal{F}_D(\Z)}(\gamma_Q \cdot (0,\infty))\cdot Q(x,1)^{k-1}+\delta_{k,D}(x^{2k-2}-1)
\end{equation}
 up to a multiplicative constant, for $k$ \emph{even} and $D$ non square. Here $\delta_{k,D}$ denotes a certain constant which depends on $k$, $D$ and is given in \cite{KZ2} (Theorem 4). Hence one can associate to the forms $f_{k,D}$ an Eichler cocycle in $\MS^{\SL_2(\mathbb{Z})}(\mathcal{P}_{2k-2})$ whose values are polynomials with rational, indeed integral, coefficients.

We now prove an analogous and more general result for $f_{k,D}^{(p)}$, in the case where $k$ is \emph{odd}. Note that the polynomial in (\ref{koh zag}) is an even polynomial and it is indeed the \emph{even} period of $f_{k,D}$. The polynomial $\bar{\kappa}_{f_{k,D}^{(p)}}\{0,\infty\}$ of Theorem \ref{theorem period}  below is instead an odd polynomial, and it is basically the \emph{odd} period of $f_{k,D}^{(p)}$.

\begin{theorem}
\label{theorem period}
If $D$ is not a square, then
$$
\bar{\kappa}_{f_{k,D}^{(p)}}\{r,s\}(x)=3\pi\sqrt{-1}{2k-2 \choose k-1}\frac{1}{D^{k-1}\sqrt{D}}\sum_{Q  \in \mathcal{F}_D^{(p)}(\Z)}(\gamma_Q\cdot (r,s))(\gamma_Q \cdot e_0)\cdot Q(x,1)^{k-1},
$$
where $\sqrt{-1}$ denotes the square root of $-1$ in $\C$.

\end{theorem}
\begin{proof}
We will start assuming $r,s\in\Q$. Let $w=2k-2$. Then
\begin{eqnarray*}
\bar{\kappa}_{f_{k,D}^{(p)}}\{r,s\}(x)&=&\int_r^sf_{k,D}^{(p)}(z)(x-z)^{2k-2}dz-\int_{-r}^{-s}f_{k,D}^{(p)}(z)(-x-z)^{2k-2}dz \\
&=&\sum_{i=0}^{w}{w \choose i}x^{w-i}\Big[  \int_r^s(-z)^{i}\cdot f_{k,D}^{(p)}(z)dz-\int_{-r}^{-s}z^{i}\cdot f_{k,D}^{(p)}(z)dz\Big] \\
&=&\sum_{i=0}^{w}{w \choose i}x^{w-i}\Big[\sum_{Q  \in \mathcal{F}_D^{(p)}(\Z)}\int_r^s\frac{(\gamma_Q\cdot e_0) \cdot(-z)^i}{Q(z,1)^k}  -  \sum_{Q  \in \mathcal{F}_D^{(p)}(\Z)}\int_{-r}^{-s}\frac{(\gamma_Q\cdot e_0) \cdot z^i}{Q(z,1)^k} \Big].
\end{eqnarray*}
The last equality holds because the series defining $f_{k,D}^{(p)}$ converges absolutely uniformly on suitable compact sets containing the semicircle joining $r$ to $s$. Now note that this also implies that the series appearing in the last expression converge absolutely, hence we can rewrite the expression as
\begin{eqnarray*}
& &\sum_{i=0}^{w}{w \choose i}x^{w-i}\Big[\sum_{Q  \in \mathcal{F}_D^{(p)}(\Z)}(\gamma_Q\cdot e_0) \cdot \Big (\int_r^s\frac{(-z)^i}{Q(z,1)^k}  - \int_{-r}^{-s}\frac{z^i}{Q(z,1)^k}\Big ) \Big]\\
&=& \sum_{i=0}^{w}{w \choose i}x^{w-i}\Big[\sum_{Q  \in \mathcal{F}_D^{(p)}(\Z)}(\gamma_Q\cdot e_0) \cdot \Big (\int_r^s\frac{(-z)^i}{Q(z,1)^k}  - \int_{-r}^{-s}\frac{z^i}{\widetilde{Q}(z,1)^k}\Big ) \Big],
\end{eqnarray*}
where if $Q=[a,b,c]$ then $\widetilde{Q}:=[-a,b,-c]$. The last equality holds because we can rearrange the terms of a series which converges absolutely. By Proposition \ref{bigThm} and Lemma \ref{lemma claim} below, we can rewrite the difference of the integrals above as
\begin{eqnarray*}
& &3\pi\sqrt{-1}\bigg( (-1)^i\frac{(\gamma_Q\cdot(r,s))\cdot\Aui+G_{(r,s)}^{Q,i}}{a^k}-\frac{ (\gamma_{\widetilde{Q}}\cdot(-r,-s))\cdot A_{\widetilde{Q},1}^{(i)}+G_{(-r,-s)}^{\widetilde{Q},i} }{(-a)^k}\bigg)\\
&=&2(-1)^ia^{-k}3\pi\sqrt{-1}\cdot (\gamma_Q\cdot(r,s))\cdot\Aui.
\end{eqnarray*}
For $r,s\in\Q$ the Theorem then follows by formally comparing this expression with Theorem \ref{theorem residue}.

The case $\{r,s\}=\{0,\infty\}$ follows from Theorem 5 of \cite{KZ2}. Indeed for any cusp form $f$, the polynomial $\bar{\kappa}_f\{0,\infty\}(x)$ is odd, and our result for $\{0,\infty\}$ can be read from the odd part of the polynomial in Theorem 5 of \cite{KZ2}.

The case $\{r,\infty\}$ follows because $\bar{\kappa}_f\{r,\infty\}(x)=\bar{\kappa}_f\{r,0\}(x)+\bar{\kappa}_f\{0,\infty\}(x)$.
\end{proof}

\begin{remark}
The period polynomials of Theorem \ref{theorem period} basically have the same expression as the polynomials $\kappa_{k,D}\{r,s\}$ defined in Proposition \ref{proof is MS}. However, the period polynomials $\bar{\kappa}_{f_{k,D}^{(p)}}\{r,s\}(x)$ have coefficients in $\C$ and the square root of $D$ appearing in their formula is the complex one. The polynomials $\kappa_{k,D}\{r,s\}$ instead have coefficients in $\C_p$ and the square root of $D$ appearing in their definition is an element of $\Q_p$.
\end{remark}
 
\vspace{4mm}

We devote the rest of this section to proving the results mentioned in the proof of Theorem \ref{theorem period}. Given a binary quadratic form $Q=[a,b,c],$ let $r_1$ and $r_2$ denote, as usual, its \emph{first} and \emph{second} root. Consider the following partial fraction decomposition
$$
\frac{z^i}{(z-r_1)^k(z-r_2)^k}=\frac{\Aui}{(z-r_1)}+...+\frac{\Aki}{(z-r_1)^k}+\frac{\Bui}{(z-r_2)}+...+\frac{\Bki}{(z-r_2)^k}.
$$
For any $r,s\in\Q$, let $G_{(r,s)}^{Q,i}$ be defined as
\begin{eqnarray*}
G_{(r,s)}^{Q,i}&:=& \sum_{l=2}^k \Big( \frac{\Ali}{(1-l)(s-r_1)^{l-1}} +  \frac{\Bli}{(1-l)(s-r_2)^{l-1}}  -\frac{\Ali}{(1-l)(r-r_1)^{l-1}} -  \frac{\Bli}{(1-l)(r-r_2)^{l-1}} \Big )\\
&+&\Aui \ln \Big | \frac{s-r_1}{s-r_2} \Big |-\Aui \ln \Big | \frac{r-r_1}{s-r_2} \Big | .
\end{eqnarray*}

\begin{proposition}
\label{bigThm}
Let $r,s\in\Q$ and assume that the discriminant of $Q$ is not a square. Then 
$$
\int_{r}^s\frac{z^i}{(z-r_1)^k(z-r_2)^k}dz=3\pi\sqrt{-1}\cdot(\gamma_Q\cdot(r,s))\cdot\Aui+G_{(r,s)}^{Q,i},
$$
where  $\sqrt{-1}$ denotes the square root of $-1$ in $\C$ and the integral is taken over the geodesic in the upper-half plane joining $r$ and $s$.
\end{proposition}
\begin{proof}
We will denote by $(r,s)$ the geodesic in the upper-half plane joining $r$ and $s$, which is a semicircle having $r$ and $s$ as endpoints and oriented from $r$ to $s$.

\vspace{4mm}

{\textbf{Case $(\gamma_Q\cdot(r,s))=0$}}

In this case either both roots lie \enquote{inside} $(r,s)$ or \enquote{outside} of it. If both $r_1$ and $r_2$ are outside the semicircle, then the Residue Theorem implies that the integral on $(r,s)$ is the same as the line integral from $r$ to $s$, which is exactly $G_{(r,s)}^{Q,i}$.

If both roots are inside $(r,s)$ we proceed similarly but this time we need to add detours around the poles. Assume also that $r<r_1<r_2<s$ to fix things. Let $\rho$ be a positive number with $4\rho<s-r$. Consider the path given by the union of the segment joining $r$ and $r_1-\rho$, the semicircle in the lower half-plane of radius $\rho$ and center $r_1$, the segment joining $r_1+\rho$ and $r_2-\rho$, the semicircle in the lower half-plane of radius $\rho$ centered at $r_2$, and the segment joining $r_2+\rho$ with $s$. Call this path $S_{\rho}$. Then the Residue Theorem implies
$$
 \int_{r}^s\fun=\int_{S_{\rho}}\fun=G_{(r,s)}^{Q,i}.
$$

\vspace{4mm}
{\textbf{Case $(\gamma_Q\cdot(r,s))\ne0$}}

In this case one root is inside the semicircle $(r,s)$ and the other is outside. Call $y$ the positive root. We will proceed similarly to the previous case, but we need to add a detour around $y$. Let the path  $S_{\rho}$ be defined as the union of the segment joining $r$ with $y-\rho$, the semicircle in the lower half-plane of radius $\rho$ centered at $y$, and the segment joining $y+\rho$ with $s$. By the Residue Theorem
\begin{equation}
\label{PIECE0}
\int_{r}^s\fun=2\pi\sqrt{-1}\cdot(\gamma_Q\cdot(r,s))\cdot\Aui+\int_{S_{\rho}}\fun.
\end{equation}
The theorem follows as one can check that
$$
\int_{S_{\rho}}\fun=\pi\sqrt{-1}\cdot(\gamma_Q\cdot(r,s))\cdot\Aui+G_{(r,s)}^{Q,i}.
$$
\end{proof}

\begin{lemma}
\label{lemma claim}
Given a binary quadratic form $Q=[a,b,c]$, let $\widetilde{Q}:=[-a,b,-c]$. Then for any $l=0,...,k$ and $i=0,...,2k-2$ we have
$$
 A_{\widetilde{Q},l}^{(i)}=(-1)^{l+i}\Ali, \:\:\:\: \:\: B_{\widetilde{Q},l}^{(i)}=(-1)^{l+i}\Bli ,
$$
and 
$$
G_{(r,s)}^{Q,i}=(-1)^{i+1}G_{(-r,-s)}^{\widetilde{Q},i}.
$$
\end{lemma}
\begin{proof}
Note that if $r_1,r_2$ are the first and second root of $Q$,  then $-r_1$ and $-r_2$ are the first and second root of $\widetilde{Q}$, respectively. The computation of $\Ali, \Bli$ is simply a residue computation, indeed
$$
\Ali=\Res_{r_1}\Big(\frac{z^i(z-r_1)^{l-1}}{(z-r_1)^k(z-r_2)^k} \Big)=\sum_{j=0}^{i}{i \choose j}r_1^{i-j}\Res_{r_1}\Big(\frac{(z-r_1)^{j-k+l-1}}{(z-r_2)^k}\Big).
$$
But
$$
\frac{(z-r_1)^{j-k+l-1}}{(z-r_2)^k}=\frac{(z-r_1)^{j-k+l-1}}{(r_1-r_2)(k-1)!} \sum_{t=k-1}^{\infty}\frac{t(t-1)...(t-k+2)(z-r_1)^{t-k+1}}{(r_2-r_1)^t},
$$
hence
$$
\Res_{r_1}\Big(\frac{(z-r_1)^{j-k+l-1}}{(z-r_2)^k}\Big)={2k-j-l-1 \choose k-1}\frac{(-1)}{(r_2-r_1)^{2k-j-l}}.
$$
It follows that
$$
\Ali=\sum_{j=0}^{i}{i \choose j}r_1^{i-j}{2k-j-l-1 \choose k-1}\frac{(-1)}{(r_2-r_1)^{2k-j-l}}.
$$
Now $A_{\widetilde{Q},l}^{(i)}$ can be found from the formula for $\Ali$ simply by substititing $-r_1$ and $-r_2$ in place of $r_1,r_2$. Then it is clear that $A_{\widetilde{Q},l}^{(i)}=(-1)^{l+i}\Ali$. The relationship for $\Bli$, $B_{\widetilde{Q},l}^{(i)}$ can be found in a similar way and the Lemma follows immediately.
\end{proof}

\vspace{4mm}

\section{A Shimura-Shintani correspondence for rigid analytic cocycles of higher weight}
\label{main juice}
The aim of this section is to construct a cusp form $\hat{\Omega}_k(q)$ of weight $k+1/2$ and level $4p^2$ with coefficients in $\MS^{\Gamma}(\mathcal{A}_{2k})$, for $k\geq 3$ odd. 
More precisely, $\hat{\Omega}_k(q)$ should be an element of $S_{k+1/2}^{(\bar{\Q})}(\Gamma(4p^2))\otimes \MS^{\Gamma}(\mathcal{A}_{2k})\subset \bar{\Q}[[q]]\otimes \MS^{\Gamma}(\mathcal{A}_{2k})$, where $S_{k+1/2}^{(\bar{\Q})}(\Gamma(4p^2))$ are the weight $k+1/2$ cusp forms of level $4p^2$ whose Fourier coefficients are in $\bar{\Q}$.

We will now describe how one can get a correspondence
$$
 S_{k+1/2}^{(\bar{\Q})}(\Gamma(4p^2)) \:\xlongrightarrow{\mathcal{C}} \:\MS^{\Gamma}(\mathcal{A}_{2k})
$$
via $\hat{\Omega}_k(q)$. Let $\{ g_1,...,g_t \}$ be a basis of eigenforms for $S_{k+1/2}^{(\bar{\Q})}(\Gamma(4p^2))$. Then we can write
$$
\hat{\Omega}_k(q)=\sum_{i=1}^tg_i\otimes m_i, \:\:\:\: \text{ for some $m_i\in \MS^{\Gamma}(\mathcal{A}_{2k})$}.
$$
Now let $g\in S_{k+1/2}^{(\bar{\Q})}(\Gamma(4p^2))$ and write $g=\sum_{i=1}^t\alpha_ig_i$ for some $\alpha_i\in\bar{\Q}$. Then the correspondence $\mathcal{C}$ is defined by letting
$$
\mathcal{C}: \:g \mapsto \sum_{i=1}^t\alpha_i m_i.
$$

As mentioned in the introduction, the rigid analytic cocycle $J_{k,D}$ should play for the correspondence that we aim to construct a role analogous to the role played by the Zagier form $f_{k,D}$ for the classical Shimura-Shintani correspondence. In particular, the series $\hat{\Omega}_k(q)$ should have an expression of the form $\hat{\Omega}_k(q)=\sum_{D>0}D^{k-1/2}J_{k,D}\cdot q^D$, which mimics the formula for the holomorphic kernel function for the Shimura-Shintani correspondence $\Omega_k(q)=\sum_{D>0}D^{k-1/2}f_{k,D}\cdot q^D$. However, if $D$ is a square then $J_{k,D}$ is not defined, so our main result is slightly different. 

\begin{theorem}
Let $k\ge 3$ be odd. If $D$ is not a square and $\big (\frac{D}{p}\big)= 1$, then $D^{k-1/2}J_{k,D}$ is the $D$-th coefficient of a weight $k+1/2$ cusp form $\hat{\Omega}_k(q)$ of level $4p^2$ with coefficients in $\MS^{\Gamma}(\mathcal{A}_{2k})$. The $D$-th coefficient of $\hat{\Omega}_k(q)$ vanishes if $\big (\frac{D}{p}\big)\ne 1$. 
\end{theorem}
\begin{proof}
The proof consists of two steps.

\vspace{2mm}

\noindent
\emph{Step 1.}

\noindent
We will at first construct a level $4p^2$ cusp form  $\bar{\Omega}_k(q)=\sum_{D>0}c_D\cdot q^D$ of weight $k+1/2$ with coefficients $c_D\in\mathfrak{S}_{2k}^{(p)}(\bar{\Q})$,   where $\mathfrak{S}_{2k}^{(p)}(\bar{\Q})\subset S_{2k}(\Gamma_0(p))$ is a certain vector space over $\bar{\Q}$. More precisely, $\bar{\Omega}_k(q)$ will be an element of $S_{k+1/2}^{(\bar{\Q})}(\Gamma(4p^2))\otimes \mathfrak{S}_{2k}^{(p)}(\bar{\Q})\subset \bar{\Q}[[q]]\otimes \mathfrak{S}_{2k}^{(p)}(\bar{\Q})$. The $D$-th coefficient of $\bar{\Omega}_k(q)$ vanishes if $\big (\frac{D}{p}\big)\ne 1$.
This construction will be carried out in Theorem \ref{complex generating series}.

\vspace {2mm}

\noindent
\emph{Step 2.}

\noindent
We will then construct a $\bar{\Q}$-linear map $\iota:\mathfrak{S}_{2k}^{(p)}(\bar{\Q})\rightarrow \MS^{\Gamma}(\mathcal{A}_{2k})$ and show that $c_D\mapsto D^{k-1/2}J_{k,D}$ if $D$ is not a square. By definition, the resulting generating series $\hat{\Omega}_k(q)=\sum_{D>0}\iota(c_D)\cdot q^D$ is a weight $k+1/2$ cusp form of level $\Gamma(4p^2)$ with coefficients in $\MS^{\Gamma}(\mathcal{A}_{2k})$, i.e. an element of an element of $S_{k+1/2}^{(\bar{\Q})}(\Gamma(4p^2))\otimes \MS^{\Gamma}(\mathcal{A}_{2k})\subset \bar{\Q}[[q]]\otimes \MS^{\Gamma}(\mathcal{A}_{2k})$. This will be done in Theorem \ref{linear maps}.
\end{proof}

Fix an embedding of $\bar{\Q}$ into $\C_p$ and let $\mathcal{P}_{2k-2}(\bar{\Q})\subset \mathcal{P}_{2k-2}$ be the polynomials with coefficients in $\bar{\Q}$ and degree at most $2k-2$. Let  $\mathfrak{S}_{2k}^{(p)}(\bar{\Q})\subset S_{2k}(\Gamma_0(p))$ be the subset of forms $f$ such that $(3\pi\sqrt{-1})^{-1}\cdot\bar{\kappa}_f\in\MS^{\Gamma_0(p)}(\mathcal{P}_{2k-2}(\bar{\Q}))$. This is a vector space over $\bar{\Q}$ containing $f_{k,D}^{(p)}$.

\begin{theorem}
\label{complex generating series}
Let $k\ge 3$ be odd. Consider the series $\bar{\Omega}_k(q)=\sum_{D>0}D^{k-1/2}f_{k,D}^{(p)}\cdot q^D$, where $D$ ranges over discriminants with $(\frac{D}{p})=1$. Then $\bar{\Omega}_k$ is a weight $k+1/2$ cusp form of level $4p^2$ with coefficients in $\mathfrak{S}_{2k}^{(p)}(\bar{\Q})$.
\end{theorem}
\begin{proof}
We will proceed similarly to \cite{KZ1} in their Theorem 2, in particular we will use a theorem of \cite{Vigneras} which was generalized by Stopple (\cite{Stopple}). At first, note that the square root of $D$ appearing in the definition of $\bar{\Omega}_k$ is the positive one in  $\mathbb{R}$, and that $\bar{\Omega}_k$ is well defined if we choose a set of representatives $\mathbb{F}_p^+$ for $\mathbb{F}_p^{\times}/\{\pm 1\}$. Indeed, the set $\mathbb{F}_p^+$ gives a choice of $\sqrt{D}\in\Q_p$ as $D$ varies, and this implies that there is no ambiguity in the choice of the terms $(\gamma_Q\cdot e_0)$ appearing in the definition of $f_{k,D}^{(p)}$. For $\big (\frac{D}{p}\big)= 1$, let $s,-s$ be the square roots of $D$ $(mod \:p)$ and let
$$
\mathcal{F}_D^{(p),s}:=\{ Q=[a,b,c]\in \mathcal{F}_D^{(p)} \text{ such that } b\equiv s \:\:(mod \: p) \}.
$$
Then note that
$$
f_{k,D}^{(p)}(z)=\sum_{Q\in \mathcal{F}_D^{(p)}(\Z)}\frac{(\gamma_Q\cdot e_0)}{Q(z,1)^k}=\pm 2\bigg(\sum_{Q\in \mathcal{F}_D^{(p),s}(\Z)}\frac{1}{Q(z,1)^k}\bigg),
$$
where the sign depends on the choice of a square root of $D$ in $\Q_p$. By letting $q=e^{2\pi i\tau}$ we can see $\bar{\Omega}_k$ as a function of two variables $z,\tau\in\mathcal{H}$ as 
\begin{equation}
\label{omega bar}
\bar{\Omega}_k(z,\tau)=\sum_{D>0}D^{k-1/2}f_{k,D}^{(p)}(z)e^{2\pi iD\tau}=\sum_{\alpha \in \mathbb{F}_p^+}\bar{\Omega}_{k,\alpha}(z,\tau),
\end{equation}
where
\begin{equation}
\bar{\Omega}_{k,\alpha}(z,\tau):=\sum_{\substack{ [a,b,c]\in \mathfrak{H}_{\alpha} \\b^2-4ac>0  }}\frac{(b^2-4ac)^{k-\frac{1}{2}}}{(az^2+bz+c)^k}\cdot e^{2\pi i(b^2-4ac)\tau}
\end{equation}
and $\mathfrak{H_{\alpha}}$ is the coset of $[0,\alpha,0]$ in $L/L^{\prime}$ with $L:=\{ [a,b,c]\in \Z^3 \text{ such that } p|a  \}$ and $L^{\prime}:=\{ [a,b,c]\in L \text{ such that } p|b  \}$. Now we apply Theorem 1 of \cite{Stopple}, which is a generalization of the theorem on p. 228 of \cite{Vigneras}. (Stopple considers only the case of lattices of even dimension, while Vigneras considers also odd dimensions. The result of Stopple still holds in our case as we are working over $\Q$, see \cite{Stopple} for more details). This shows that the functions $\bar{\Omega}_{k,\alpha}(z,\tau)$ are weight $k+1/2$ modular forms in $\tau$ whose level is the same as the level of the lattice $L^{\prime}$. By computing the dual of $L^{\prime}$, we see that the level is $4p^2$. 
Since $\bar{\Omega}_k(z,\tau)$ has no constant term, the theorem follows.
\end{proof}

\begin{remark}
Note that the functions $\bar{\Omega}_{k,\alpha}$ defined in the proof of Theorem \ref{complex generating series} are defined only up to sign. Their sign depends on the choice of the representatives $\mathbb{F}_p^{+}$. This does not affect the proof of Theorem \ref{complex generating series}.
\end{remark}

\begin{theorem}
\label{linear maps}
If $D$ is not a square, there exists a $\bar{\Q}$-linear map $\iota:\mathfrak{S}_{2k}^{(p)}(\bar{\Q})\rightarrow \MS^{\Gamma}(\mathcal{A}_{2k})$ such that $ \iota\big(f_{k,D}^{(p)}\big)= J_{k,D}$.
\end{theorem}
\begin{proof}
The map $\iota$ will be the composite $\iota=ST\circ\mathfrak{p}$ of the two maps
$$
\mathfrak{S}_{2k}^{(p)}(\bar{\Q})\xrightarrow{\mathfrak{p}}\MS^{\Gamma_0(p)}(\mathcal{P}_{2k-2}(\bar{\Q}))\subset \MS^{\Gamma_0(p)}(\mathcal{P}_{2k-2})\xrightarrow{\ST}\MS^{\Gamma}(\mathcal{A}_{2k}).
$$
Here $\mathfrak{p}$ is the $\bar{\Q}$-linear map defined as
$$
\mathfrak{p}(f):=\frac{1}{3\pi\sqrt{-1}}\cdot\bar{\kappa}_f
$$
and $\ST$ is the Schneider-Teitelbaum  lift for rigid analytic cocycles defined in Section \ref{ST}. Theorem \ref{theorem period} and Theorem \ref{theorem residue} imply that $\mathfrak{p}(f_{k,D}^{(p)})=\kappa_{k,D}$, where $\kappa_{k,D}$ was defined in proposition \ref{proof is MS}.

In Theorem \ref{theorem residue} we proved that $\Res_{0}(J_{k,D})=\kappa_{k,D}$. By Corollary 2.3.4. and Theorem 4.5.2. of \cite{Das}, the map $\Res$ defined in Section \ref{res map} is injective, and by Lemma \ref{Lemma: isom} the map $\Res_0$ is also injective. Hence $\ST(\kappa_{k,D})=J_{k,D}$ and the theorem follows.
\end{proof}

\vspace{2mm}

\bibliographystyle{amsalpha}
\providecommand{\bysame}{\leavevmode\hbox to3em{\hrulefill}\thinspace}
\providecommand{\href}[2]{#2}

\end{document}